\documentclass[12pt,dvipdfmx
]{amsart}
\usepackage{amsfonts}
\usepackage{amsmath, amsthm, amssymb,color}
\usepackage{latexsym}
\usepackage{txfonts}
\usepackage{bm}
\usepackage{graphicx}
\usepackage{graphics}
\usepackage{epsfig}
\usepackage[T1]{fontenc} 
\usepackage{mathrsfs}
\usepackage{lscape}

\usepackage{color}
\usepackage{ulem}

\numberwithin{equation}{section}
\allowdisplaybreaks

\newtheorem{lemma}{Lemma}[section]

\newtheorem{theorem}[lemma]{Theorem}
\newtheorem{proposition}[lemma]{Proposition}
\newtheorem{definition}[lemma]{Definition}
\newtheorem{corollary}[lemma]{Corollary}
\newtheorem{example}[lemma]{Example}
\newtheorem{exercise}[lemma]{Exercise}
\newtheorem{remark}[lemma]{Remark}
\newtheorem{fig}[lemma]{Figure}
\newtheorem{tab}[lemma]{Table}

\newcommand{\bth}{\begin{theorem}}
\newcommand{\ethe}{\end{theorem}}
\newcommand{\bre}{\begin{remark}\em }
\newcommand{\ere}{\end{remark}}
\newcommand{\ble}{\begin{lemma}}
\newcommand{\ele}{\end{lemma}}
\newcommand{\bde}{\begin{definition}}
\newcommand{\ede}{\end{definition}}
\newcommand{\bco}{\begin{corollary}}
\newcommand{\eco}{\end{corollary}}
\newcommand{\bpr}{\begin{proposition}}
\newcommand{\epr}{\end{proposition}}

\newcommand{\bexer}{\begin{exercise}}
\newcommand{\eexer}{\end{exercise}}

\newcommand{\bexam}{\begin{example}\rm  }
\newcommand{\eexam}{ \end{example}}

\newcommand{\bfi}{\begin{fig}}
\newcommand{\efi}{\end{fig}}

\newcommand{\btab}{\begin{tab}}
\newcommand{\etab}{\end{tab}}

\def\E{{\mathbb{E}}}
\def\P{{\mathbb{P}}}
\def\R{{\mathbb{R}}}

\def\N{{\mathbb{N}}}

\def\B_e{B_{\eta}(e)}

\newcommand{\ov}{\overline}

\newcommand{\wt}{\widetilde}

\definecolor{darkblue}{rgb}{0,0,1}
\definecolor{darkgreen}{rgb}{0,1,0}
\definecolor{darkred}{rgb}{1, 0,0}

\newcommand{\sign}{{\rm sign}}

\newcommand{\var}{{\rm var}}

\newcommand{\beao}{\begin{eqnarray*}}
\newcommand{\eeao}{\end{eqnarray*}\noindent}

\newcommand{\beam}{\begin{eqnarray}}
\newcommand{\eeam}{\end{eqnarray}\noindent}

\newcommand{\beqq}{\begin{equation}}
\newcommand{\eeqq}{\end{equation}\noindent}

\newcommand{\bce}{\begin{center}}
\newcommand{\ece}{\end{center}}

\newcommand{\barr}{\begin{array}}
\newcommand{\earr}{\end{array}}

\newcommand{\vague}{\stackrel{\lower0.2ex\hbox{$\scriptscriptstyle
                    \it{v} $}}{\rightarrow}}
\newcommand{\weak}{\stackrel{\lower0.2ex\hbox{$\scriptscriptstyle
                    \it{w} $}}{\rightarrow}}
\newcommand{\what}{\stackrel{\lower0.2ex\hbox{$\scriptscriptstyle
                    \it{\hat{w}} $}}{\rightarrow}}

\newcommand{\bdis}{\begin{displaymath}}
\newcommand{\edis}{\end{displaymath}\noindent}

\newcommand{\vep}{\varepsilon}

\newcommand{\bfa}{{\bf a}}

\newcommand{\bali}{\begin{align}}
\newcommand{\eali}{\end{align}}

\usepackage{ascmac}	
\begin{document}

\bibliographystyle{alpha}
\title[Trigonometrically approximated maximum likelihood estimation for stable law]
      {Trigonometrically approximated maximum likelihood estimation for stable law}
\today
\author[M. Matsui]{Muneya Matsui}
\address{Department of Business Administration, Nanzan University,
	18 Yamazato-cho Showa-ku Nagoya, 466-8673, Japan}
\email{mmuneya@nanzan-u.ac.jp}

\author[N. Sueishi]{Naoya Sueishi}
\address{Graduate School of Economics, Kobe University, 2-1 Rokkodai-cho, Nada-ku, Kobe, Hyogo, Japan}
\email{sueishi@econ.kobe-u.ac.jp}

\begin{abstract}
A trigonometrically approximated maximum likelihood estimation for $\alpha$-stable laws is proposed.
 The estimator solves 
 the approximated likelihood equation,
 which is obtained by projecting a true score function on the space spanned by trigonometric functions. 
 The projected score is expressed only by real and imaginary parts of the characteristic function and their derivatives, so that 
 we can explicitly construct the targeting estimating equation.   
 We study the asymptotic properties of the proposed estimator and show consistency and asymptotic normality. 
 Furthermore, as the number of trigonometric functions increases, the estimator converges to the exact maximum likelihood estimator, 
 in the sense that they have the same asymptotic law. 
 Simulation studies show that our estimator outperforms other moment-type estimators, and  
 its standard deviation almost achieves the Cram\'er--Rao lower bound. 
 We apply our method to the estimation problem for $\alpha$-stable Ornstein--Uhlenbeck processes 
 in a high-frequency setting. The obtained result demonstrates the theory of 
 asymptotic mixed normality.
\vspace{2mm}   
\end{abstract}
\thanks{Matsui's research is partly supported by the JSPS Grant-in-Aid
for Scientific Research C (19K11868).
Sueishi thanks the support by the JSPS Grant-in-Aid for Scientific Research B (19H01473)}
\maketitle

\section{Introduction}
\label{sec:introduction}
This paper proposes a novel characteristic function (ch.f.)-based estimation method called the trigonometrically approximated maximum likelihood estimation (TMLE). 
Our approach relies on the approximation method of a score function in Brant \cite{brant:1984}, 
where the approximation is done by projecting 
a true score on the functional space spanned by trigonometric functions.  
By solving the approximated score equation, we conduct TMLE.
Because the projected score is expressed by combinations of real and imaginary parts of the ch.f. 
and its derivatives, we can easily calculate the approximated score solely from the original ch.f.
This method is effective for the class of distributions whose ch.f.'s are available in closed form but whose densities are not. 
An example is the class of infinitely divisible (ID) distributions, which provides
finite dimensional distributions for fundamental stochastic processes such as L\'evy processes and additive processes 
(see, e.g. Sato \cite{sato:1999}).

Specifically we apply the proposed method to the estimation of stable laws, 
which hold a prominent position among ID distributions.  
For details and notable properties of stable laws, see Samorodnitsky and Taqqu \cite{samorodnitsky:taqqu:1994}, Nolan \cite{nolan:2020} and references therein.
We show that the trigonometrically approximated maximum likelihood (TML) estimator converges to the true maximal likelihood (ML) estimator as 
the number of trigonometric functions 
increases.
Therefore, TMLE takes over nice properties of the maximum likelihood estimation (MLE), such as asymptotic normality and efficiency. 
Furthermore, we extend TMLE to an estimation method 
suitable for the parameter estimation of $\alpha$-stable Ornstein--Uhlenbeck processes.
In what follows, we briefly 
introduce the literature and explain 
features of our estimator by comparing it with previous estimation methods.

The parameter estimation based on the ch.f. was originally proposed by Press \cite{press:1972} and has been 
developed into various directions. A regression-type estimator was proposed by 
Koutrouvelis \cite{koutrouvelis:1980,koutrouvelis:1981} and further investigated by Kogan and Williams \cite{kogon:williams:1998}. 
Paulson et al.\,\cite{paulsoetal:1975} considered an estimator that minimizes the squared integrated distance between the true ch.f. and the empirical one. 
Heathcote \cite{heathcote:1977} implemented the method and discussed its efficiency and robustness.  
Feuerverger and McDunnough \cite{feuerverger:mcDunnough:1981a, feuerverger:mcDunnough:1981b} 
showed that their moment-type estimator can attain arbitrary high efficiency. 
Notice that some of the estimation methods are applicable not only for stable 
laws but also for general distributions with explicit ch.f's. 

The estimator by \cite{feuerverger:mcDunnough:1981a, feuerverger:mcDunnough:1981b} can be viewed as
 the generalized method of moments (GMM) estimator \cite{hansen:1982}.  
In the econometric literature, the GMM estimator and its variants 
\cite{hansen:heaton:yaron:1996, imbens:spady:johnson:1998, kitamura:stuzer:1997, newey:smith:2004, qin:lawless:1994} are 
commonly used to achieve semiparametrically efficient estimation.
However, they are suboptimal for the estimation of the stable law if they use only a fixed number of moment conditions, constructed from the ch.f.
Indeed, to attain high efficiency in the GMM estimator of \cite{feuerverger:mcDunnough:1981a} 
the number of moment conditions should be increased. 
Kunitomo and Owada \cite{kunitomo:owada:2006} showed that the empirical likelihood (EL) estimator 
is asymptotically normal and efficient if the number of moment conditions increases at a certain rate with the sample size.

Although the GMM-type estimators satisfy desirable properties, they have two deficiencies.
First, the number of moment conditions can increase only at a slow rate.
For instance, in the EL estimator it must increase at the rate $n^{1/3 -\eta}$ for $\eta>0$, where $n$ is the sample size.
Thus, the EL estimator cannot achieve high efficiency in a finite sample.
Second, previous studies \cite{carrasco:florens:2002, singleton:2001, yu:2004} 
point out that the weight matrix of the optimal GMM estimator tends to be singular as the number of moment conditions increases.
Therefore, the GMM-type estimators are infeasible when the number of moment conditions is relatively large compared to the sample size.

To overcome these problems, Carrasco and Florens \cite{carrasco:florens:2000} proposed a version of the GMM estimator 
that utilizes a continuum of moment conditions (hereafter, the CGMM estimator).
See also \cite{carrasco:chrnov:florens:ghysels:2007, carrasco:florens:2002, carrasco:kotchoni:2017}. 
The CGMM uses an uncountable number of moment conditions from the beginning, 
which may cause the degeneracy of a certain  covariance operator.
However, the degeneracy problem is avoided by introducing regularization. 
Thus, the number of moment conditions is not restricted by the sample size.

TMLE addresses the same issues with CGMM.
As we will see in the next section, the TML estimator has a close relation with the optimally weighted GMM estimator.
The number of trigonometric functions in TMLE plays the same role as the number of moment conditions in GMM estimation.
Moreover, to obtain the approximated score function, we calculate a covariance matrix that is similar to the weight matrix of the optimal GMM.
The novelty of TMLE is that the number of trigonometric functions can be increased regardless of the sample size to satisfy asymptotic normality and efficiency.
Moreover, the covariance matrix of TMLE is non-singular even if its dimension is large.
Furthermore, our simulation shows that the TML estimator outperforms the CGMM estimator in terms of finite sample efficiency.

As for the exact MLE, the asymptotic theory has been established earlier by DuMouchel \cite{dumouchel:1973} 
and recently elaborated by Matsui \cite{matsui:2019} under continuous parametrizations.
Since the ch.f. is the Fourier transform of the density, the likelihood function can be obtained by numerically 
evaluating the inversion formula. However, one encounters several difficulties in practical implementation. 
So far the methods of implementation have been studied only intermittently. 
Started from an early work by DuMouchel \cite{dumouchel:1975}, several attempts have been made,  
especially in the symmetric case (see references in Nolan \cite[Section 4]{nolan:2020}). 
However, in the general non-symmetric case, there are only a few studies \cite{mittniketal:1999,nolan:2001}. 

There are two issues in implementing MLE: computational complexity and  
accuracy. 
There exists a clear trade-off relation between these two aspects. 
For example, to ease the computational burden of the inversion formula for the densities, 
Mittnik et al. \cite{mittniketal:1999} have examined MLE solely by the FFT-based approximation of density functions. 
However, that limits the accuracy, especially in the tail. 
Nolan \cite{nolan:2001} uses a precomputed spline approximation of stable densities in doing MLE, which is equipped as 
the program ``STABLE''. The program is fast, but the values reported by it have limited accuracy in several regions due 
to technical difficulties in approximation (see also \cite{royuela:at:al:2017}).
On the other hand, 
regarding the pursuit of accuracy, 
few papers have rigorously validated the accuracy of their implemented ML estimators. Assuming symmetry, Matsui \cite{matsui:takemura:2008} improved 
the accuracy of ML estimators and certified the accuracy by 
comparing the true Fisher information matrix with the empirical one calculated 
from simulated ML estimators, and have observed that the two are almost identical.
To obtain high accuracy, a finite integral representation of Zolotarev \cite{zolotarev} is used for the central part of the density and an infinite series expansion of the density is applied for the tail part. 
Although the procedure is highly accurate, its computational burden is quite large. Besides one has to 
determine the regions where the finite integral representation or the series expansion is appropriate.  

Turning now to TMLE, since it uses the ch.f. and its derivatives only, we do not need numerical integrations nor 
series expansions of densities, so that it is computationally simple and light. 
Furthermore, it gives high accuracy.  
Our simulation result shows that the variance of the TML estimator almost achieves the Cram\'er--Rao lower bound.
Thus, TMLE shares nice properties with MLE such as asymptotic efficiency, while it is computationally reasonably tractable.  
Finally and most importantly, the procedure given in TMLE is easy to extend to other contexts. 
Here we apply it to approximate the conditional ML estimator of $\alpha$-stable Ornstein--Uhlenbeck (OU) processes. 
Applications to ID distributions other than stable laws are also possible (cf. \cite{sueishi:nishiyama:2005}).  
Since the class of ID distributions constitute marginal laws of L\'evy processes, we may consider, for instance
OU processes driven by L\'evy processes, which have been studied intensively in recent years. 
In summary, the method of TMLE has potential application in various active areas. 

The paper is organized as follows.
After introducing several preliminaries in Section \ref{sec:preliminary},
 we rigorously define TMLE in Section \ref{subsec:qscore}. 
The non-degeneracy of the covariance matrix for TMLE is also proved in this section.
TMLE is compared with other related estimation methods in Section \ref{subsec:comparison}. 
Asymptotics such as consistency and asymptotic normality are obtained in Section \ref{sec:asymptotics}. 
Subsection \ref{subsec:asymptotic} treats the case where the number of trigonometric functions is fixed. 
As the number of trigonometric functions increases, the TML estimator is shown to converge to the ML estimator regardless of sample size (see 
Subjection \ref{subsec:asymptQMLinfty}), so that TMLE is proved to have asymptotic efficiency. 
Most proofs and auxiliary results for Section \ref{sec:asymptotics} are given in the Appendix.
We briefly discuss efficiency in finite samples at the end of Section \ref{sec:asymptotics}. 
In Section \ref{sec:numerical} we examine finite sample performance of TMLE by Monte Carlo simulation, which shows that the TML estimator outperforms other asymptotically efficient estimators. 
Indeed, we see that the standard deviation of TMLE almost achieves the Cram\'er--Rao bound.
Extending the TMLE method, we consider the parameter estimation for $\alpha$-stable Ornstein--Uhlenbeck processes in 
Section \ref{application:ou}, where we numerically observe that 
the asymptotic mixed normality holds for the mean reversion parameter 
and asymptotic normality holds for other parameters in a high-frequency setting.

\section{Trigonometrically approximated maximum likelihood estimation}
\label{sec:qscore}
\subsection{Preliminary}
\label{sec:preliminary} 
We work on a continuous parametrization of the stable law (cf. \cite{zolotarev,nolan:1998}) whose ch.f. is 
given by 
\begin{align}
\label{chfM}
\begin{split}
 \varphi(u;\theta) 
= \left \{
\begin{array}{ll}
\exp\Big(
 -|\sigma u|^\alpha \big\{
 1+i\beta\, \sign u \tan \frac{\pi\alpha}{2}(|\sigma u|^{1-\alpha}-1)
\big\} +i\mu u 
\Big)  & \text{if}\ \alpha\neq 1\\
 \exp\Big(
-|\sigma u| -i  \sigma u \, (2\beta /\pi) \log |\sigma u| +i\mu u
\Big)  & \text{if}\ \alpha= 1,
\end{array}
\right.
\end{split}
\end{align}
where $\mu\in \R,\,\sigma \in (0,\infty),\ \alpha\in(0,2]$, and $\beta\in[-1,1]$ are parameters.
The parameter vector is denoted by $\theta=(\theta_1,\theta_2,\theta_3,\theta_4)^\top=(\mu,\sigma,\alpha,\beta)^\top$.
Moreover, we denote the parameter space 
and its interior by $\Theta$ and $\mathring{\Theta}$, respectively. 
Because the expression \eqref{chfM} 
is continuous in $\alpha$, the case
$\alpha=1$ may not be necessary.  
By the inversion formula, we see that the density $f(x;\theta)$ satisfies the condition for a location-scale family. 

We use the following notations throughout. Denote real and imaginary parts of 
the ch.f. $\varphi(u;\theta)$ respectively by $\varphi^R(u;\theta)$ 
and $\varphi^I(u;\theta)$, that is,
$\E [e^{iuX}]=\E  [\cos (uX)]+i \E [\sin (uX)]=\varphi^R(u;\theta)+i\varphi^I(u;\theta)$. 
As usual $f'$ and $f''$ 
denote the first and the second derivatives of $f$ with respect
to (w.r.t.) $x$, respectively. 
Moreover, $f_\theta=(f_{\theta_1},f_{\theta_2},f_{\theta_3},f_{\theta_4})^\top$
denotes 
the vector of partial derivatives of $f$ w.r.t. 
$\theta$. 
The second-order partial derivatives w.r.t. $x$ and $\theta$
are denoted by  
\[
 f_{\theta_i}'=\frac{\partial^2 f}{\partial x \partial \theta_i
      }=\frac{\partial^2 f}{\partial \theta_i \partial x },\quad
 f_{\theta_i\theta_j}= \frac{\partial^2 f}{\partial \theta_i \partial
      \theta_j} = \frac{\partial^2 f}{\partial \theta_j \partial
      \theta_i},\quad 
      i,j=1,\ldots,4,  
\]
that is, all derivatives treated are interchangeable in our
case (cf. \cite[Section 2]{matsui:2019}). 

\subsection{Trigonometrically approximated likelihood estimation}
\label{subsec:qscore}
As described in the introduction, we utilize the approximated score function proposed by Brant \cite{brant:1984},
which we call the trigonometrically approximated score function (TSF for short). 
By solving the approximated score equation by TSF, we define TMLE.
The key idea is that TSF is the projection of the true score function onto a
subspace spanned by trigonometric functions. 
The projected score has a nice expression given solely by the ch.f. and its derivatives. 

In what follows, we precisely explain the procedure. 
Let $L^2(f)$ be a Hilbert space with inner product
\begin{align}
\label{eq:innerprod}
 \langle h_1,h_2 \rangle = \int_{-\infty}^{\infty} h_1(x) \bar h_2(x)
 f(x;\theta) dx = \E_\theta [h_1(X)\bar h_2(X)],\quad h_1,h_2 \in L^2(f).
\end{align}
Notice that the expectation is taken w.r.t. $f(x;\theta)$, the density of $X$,
and both $\langle \cdot,\cdot \rangle$ and $L^2(f)$ depend on $\theta$; for convenience we abbreviate $\theta$ in these notations. 
For a given set of points $(u_1, \dots, u_k)$ such that $|u_i| \neq |u_j|$ for $i \neq j$, we prepare a vector of trigonometric functions    
\[
 g (x)=(\cos (u_1x),\cdots,\cos(u_k x),\sin (u_1 x),\cdots,\sin(u_k x))^\top
\]
and its expectation 
\begin{align}
\label{def:gamma}
 \gamma(\theta):= \E_\theta  [g(X)]= 
(\varphi^R(u_1;\theta),\cdots,\varphi^R(u_k;\theta), \varphi^I(u_1;\theta),\cdots,\varphi^I(u_k;\theta))^\top. 
\end{align}
The elements of vector $g$ are linearly
independent 
and span a subspace of $L^2(f)$. Generally, $g$ constitutes a non-orthogonal basis. 

Next we project the score
$S(x;\theta)=f_\theta(x;\theta)/f(x;\theta)$ onto the subspace 
spanned by centered trigonometric functions $g-\gamma(\theta)$.  
The projected score $\wt S(x;\theta)$ 
satisfies 
\begin{align*}
 S(x;\theta) &= \underbrace{A\, (g(x)-\gamma(\theta))}_{:=\wt S(x;\theta)}+( g(x)-\gamma(\theta) )^\perp   
\end{align*}
with some coefficient $A\,(4\times 2k\ \mathrm{matrix})$,
where $( g(x)-\gamma(\theta) )^\perp$ denotes the orthogonal complement of the space spanned by $\{1,g\}$. 
By taking the inner product of both sides with $g(x)-\gamma(\theta)$, we obtain 
\begin{align*}
A &= \boldsymbol{\big\langle} S, g-\gamma(\theta) \boldsymbol{\big\rangle}\, \boldsymbol{\big\langle} g-\gamma(\theta), g-\gamma(\theta)  \boldsymbol{\big\rangle}^{-1}
\end{align*}
with
\begin{align*}
& \boldsymbol{\big\langle} g-\gamma(\theta), g-\gamma(\theta)  \boldsymbol{\big\rangle}
 =\E_\theta[(g(X)-\gamma(\theta))\, (g(X)-\gamma(\theta))^\top ]=\Sigma(\theta) \\ 
&\text{and}\quad \boldsymbol{\big\langle} S, g-\gamma(\theta) \boldsymbol{\big\rangle}
  = \E_\theta [S(X;\theta)\,(g(X)-\gamma(\theta))^\top ] = \gamma_\theta(\theta), 
\end{align*}
where the inner product 
of vectors is taken 
elementwise. 
Here under exchangability of expectation and derivative, 
$\gamma_\theta (\theta)$ is expressed by a
 $4 \times 2k$ derivative matrix
\footnote{This is denominator-layout notation (the Hessian formulation, the gradient) where a vector-by-vector derivative is given by 
the transpose of Jacobian.}:
\begin{align}
\label{def:deriv:gamma}
 \gamma_\theta (\theta):=
 \frac{\partial \gamma(\theta)}{\partial \theta} = 
(\varphi^R_\theta (u_1;\theta),\ldots,\varphi^R_\theta (u_k;\theta),\varphi^I_\theta(u_1;\theta),\dots,\varphi^I_\theta (u_k;\theta)).
\end{align}
The elements of $\Sigma(\theta)$ are given by 
\begin{align}
\label{def:cov:chf}
\text{Cov}_{\theta}(\cos(u_i X),\cos (u_j X)) &= \frac{1}{2} \big( \varphi_\theta^R(u_i+u_j;\theta)+ \varphi^R_\theta(u_i-u_j;\theta)\big) -\varphi_\theta^R(u_i;\theta)\varphi^R_\theta(u_j;\theta), \nonumber\\
\text{Cov}_{\theta}(\cos(u_i X),\sin (u_j X)) &= \frac{1}{2} \big( \varphi_\theta^I(u_i+u_j;\theta)- \varphi^I_\theta(u_i-u_j;\theta) \big) -\varphi_\theta^R(u_i;\theta)\varphi^I_\theta(u_j;\theta),\\
\text{Cov}_{\theta}(\sin (u_i X),\sin (u_j X)) &= \frac{1}{2} \big( - \varphi_\theta^R(u_i+u_j;\theta)
+ \varphi^R_\theta(u_i-u_j;\theta) \big) -\varphi_\theta^I(u_i;\theta)\varphi^I_\theta(u_j;\theta). \nonumber
\end{align}
Thus TSF is constructed by 
\begin{align}
\label{def:TMLE}
 \wt S (x;\theta) = \gamma_\theta(\theta)\, 
 \Sigma(\theta)^{-1} (g(x)-\gamma(\theta)). 
\end{align}
It is easy to see that
\begin{align}
\label{hessian:tmle}
\wt I (\theta) = \var_\theta (\wt S(X;\theta)) = 
\gamma_\theta(\theta) \Sigma(\theta)^{-1}
\gamma_\theta(\theta)^\top, 
\end{align}
which will be shown to converge to the true Fisher information matrix as $k \to \infty$ in 
the proof of Theorem \ref{thm:consistency}.

Now we write the empirical TSF by 
\begin{align*}
\wt S_n(\theta) = \frac{1}{n}\sum_{j=1}^n \wt S(X_j;\theta)   
\end{align*}
and define the TML estimator $\hat \theta_n$ as an element of the set
\begin{align}
\label{def:tmle}
\{\theta\in C\, \mid\, \wt S_n(\theta)=0\}
\end{align}
with some compact set $C \subset  \mathring{\Theta}$.  
Since the existence of exact zeros of the empirical TSF $\wt S_n(\theta)$  
is not always assured in a small sample, 
if there are no roots, we choose one of local minimum values of $|\wt S_n(\theta)|$ as the estimator.
Although there might exist multiple local minima, we can appropriately choose a consistent sequence, as we will discuss shortly.

\begin{remark}
	\label{rem:sec:triscore}
	$(\mathrm{i})$ Interestingly the expectation of the derivative of $\wt S(X;\theta)$ w.r.t. $\theta$ 
	satisfies
	\[
	\E_\theta \left[-\frac{\partial \wt S(X;\theta)}{\partial \theta} \right] = \gamma_\theta(\theta) \Sigma(\theta)^{-1}
	\gamma_\theta(\theta)^\top =\wt I (\theta)
	\]
	under certain regularity conditions. The relation corresponds to the second definition of the Fisher information matrix. 
	\\
	$(\mathrm{ii})$ As one can see, the procedure of TMLE is applicable to any laws with explicit ch.f.'s. 
	Thus TMLE could be a powerful statistical tool for the class of ID distributions, 
whose ch.f is given by the L\'evy--Khintchine representation $($e.g. \cite[Section 8]{sato:1999}$)$. 
\end{remark}
Finally, we discuss the 
non-degeneracy of $\wt I (\theta)$. 
One may wonder if an arbitrary choice of evaluation points of $g$ 
results in the singularity of $\wt I$. 
In fact, we avoid this by the condition $|u_i|\neq |u_j|$ for $i\neq j$. Recall that  
$\wt S(x;\theta)$ 
is constructed by the projection on a space spanned by trigonometric functions. 
The idea behind this is to choose $(u_i)_{i\le k}$ such that trigonometric functions keep linear independence. 
In the following lemma we show that the rows of $\gamma_\theta(\theta)$ are linearly independent for $k\ge 2$ 
and the inverse of the Grammian matrix $\Sigma(\theta)^{-1}$ is positive definite. 
\begin{lemma}
\label{lem:auxiliary1}
Suppose that $u_i \neq 0$ for all $i = 1, \dots, k$.
For any $\theta \in \mathring{\Theta}$
and for any choice of  $(u_1,\ldots,u_k)\in \R^k$ such that $|u_i|\neq |u_j|$ for $i\neq j$, \\ 
$(\mathrm{i})$ the rows of $\gamma_\theta(\theta)$ are linearly independent for
$k\ge 2$, namely $\rm{rank}(\gamma_\theta(\theta))=4$. \\
$(\mathrm{ii})$ $\Sigma(\theta)$ is positive definite. 
\end{lemma}
The proof is give in Appendix \ref{sebsec:a1}. 
The result is crucial for the asymptotic theory (see the next section) and for the optimization procedure to obtain the TML estimator 
(Section \ref{sec:numerical}).

\subsection{Comparison with GMM and other related estimators}
\label{subsec:comparison}
The TML estimator has a close similarity to the GMM estimator.
Based on moment conditions $\E_{\theta_0}[g(X) -\gamma(\theta_0)]=0$, the optimally weighted two-step GMM estimator is defined by
\begin{align}
\label{def:gmm}
\hat{\theta}_{gmm} = \arg \min_{\theta \in C} \left(\frac{1}{n} \sum_{i=1}^n (g(X_i)-\gamma(\theta)) \right)^\top
\hat{\Sigma}^{-1} \left( \frac{1}{n} \sum_{i=1}^n (g(X_i)-\gamma(\theta)) \right), 
\end{align}
where $\hat{\Sigma}$ is a consistent estimator for $\Sigma(\theta_0)$.
Typically, $\hat{\Sigma}$ is obtained by
\begin{align}
\label{cov}
\hat{\Sigma}(\tilde{\theta}) = \frac{1}{n} \sum_{i=1}^n ( g(X_i) -\gamma(\tilde{\theta}) )( g(X_i) - \gamma(\tilde{\theta}))^\top,
\end{align}
where $\tilde{\theta}$ is a preliminary consistent estimator of $\theta_0$. 
The first-order condition of the minimization problem yields 
\begin{align}
\label{def:GMM:firstcondi}
\frac{1}{n} \sum_{j=1}^n \gamma_\theta(\hat{\theta}_{gmm}) \hat{\Sigma}^{-1}( g(X_i) -  \gamma(\hat{\theta}_{gmm}) ) =0.
\end{align}
Comparing with \eqref{def:TMLE}, we see that the two estimators are closely related, although their derivations are quite different: the TML estimator is obtained by the projection of the true score on trigonometric functions, 
while the GMM estimator is constructed based on the moment conditions.
One clear difference is that we treat $\Sigma(\theta)$ as a function of $\theta$ in the estimator of TMLE \eqref{def:TMLE}, and do not 
use a predetermined matrix.

Some studies pointed out that the GMM estimator is infeasible when the grid of $(u_1, \dots, u_k)$ is too fine, 
because the weight matrix becomes singular \cite{carrasco:chrnov:florens:ghysels:2007, carrasco:florens:2002, singleton:2001, yu:2004}.
There might be some confusion regarding this argument.
Even though the GMM estimator is infeasible if (\ref{cov}) is used, $\Sigma(\theta_0)$ can be estimated by $\Sigma(\tilde{\theta})$ because 
the explicit form of 
$\Sigma(\theta)$ is available 
for the stable law, cf. \eqref{def:cov:chf}.
We do not need to estimate the variance by its sample analog. 
We call the GMM estimator using $\Sigma(\tilde{\theta})$ as the weight matrix  the explicit GMM estimator. 
Lemma \ref{lem:auxiliary1} shows that $\Sigma(\tilde{\theta})$ is non-singular as long as evaluation points are properly chosen.
Therefore, the explicit GMM estimator is indeed feasible no matter how fine the grid is.
Notice, however, that the non-degeneracy of the weight matrix does not imply that explicit GMM can utilize an arbitrary large number of moment conditions to satisfy desirable asymptotic properties.
We will discuss this issue in the next section (Remark \ref{rem:asympt:expgmm}).

We explain other related estimators, which are compared with TML estimator in the simulation. 
The asymptotics of these estimators are discussed in the next section. \\
\noindent
{\bf Other related estimators} \\
$(\mathrm{i})$ 
The CGMM estimator is based on a continuum (continuous function) of moment conditions
$ \E_{\theta_{0}} [h(u,X;\theta_0)]=0$, where $h(u,X;\theta)=e^{iu X}-\varphi(u;\theta)$ 
(see \cite{carrasco:florens:2000, carrasco:florens:2002} for a detailed definition). 
More precisely, for a suitably chosen covariance operator $K$, it minimizes $\|K^{-1/2} h_n(u;\theta)\|$ 
w.r.t. $\theta$, where $\|\cdot\|$ is a norm 
and $h_n(u;\theta)=\sum_{j=1}^n h(u,X_j;\theta)$. 
Thus the CGMM estimator can be viewed as a continuous version of \eqref{def:gmm} (notice that 
$K$ plays a role of $\Sigma$). The idea behind is to pursuit efficiency by capturing more information (\cite[Sec. 5.2]{carrasco:florens:2000}). 
A problem of CGMM is that an estimate of optimal $K$ is not invertible in general.
This is dealt with a regularization parameter introduced in the estimation of $K^{-1/2}$. \\ 
$(\mathrm{ii})$ 
The EL estimator \cite{kunitomo:owada:2006} also relies on the moment conditions $\E_{\theta_0}[g(X) -\gamma(\theta_0)]=0$ and 
maximizes the empirical likelihood subject to the moment conditions. Similar to the GMM estimator 
it suffers from a severe problem of degeneracy if moment conditions are too many, because it is a one-step estimator that implicitly 
estimates the optimal weight matrix. 
We cannot use $\Sigma(\tilde{\theta})$ as the weight matrix for the EL estimator.
\begin{remark}
The models considered by \cite{singleton:2001} and \cite{carrasco:chrnov:florens:ghysels:2007} are different from ours.
 They treat more complicated financial models such as affine price models. 
 It is not certain if $\Sigma(\tilde{\theta})$ is non-singular for distributions other than the stable law when the grid of evaluation points  is fine. 
  Moreover, the condition of our lemma is violated if symmetric grid points $(-u_k,u_k)_{k\ge 1}$ are selected as in \cite{singleton:2001}. 
\end{remark}

\section{Asymptotics of trigonometric approximated maximum likelihood estimation}
\label{sec:asymptotics}
This section investigates asymptotic properties of TMLE.
The asymptotics of TMLE 
with arbitrary fixed evaluation points are studied in Subsection \ref{subsec:asymptotic}.
In Subsection \ref{subsec:asymptQMLinfty} 
we consider the case where the number of points increases. We
show that the TML estimator 
converges to the ML estimator 
in the sense that they have the same 
asymptotic law. The asymptotics of other related estimators are discussed at the end of this section.

\subsection{Asymptotics of TMLE when the number of points is fixed}
\label{subsec:asymptotic}
The TML estimator with fixed evaluation points could be regarded as a $Z$-estimator, and we exploit theorems from
van der Vaart \cite[Theorems 5.41 and Theorem 5.42]{vandervaat:2000}. 
The following is the first main result. 

\begin{theorem}
\label{thm:asymptocs:fixedk}
Let $C$ be any compact subset of $\mathring{\Theta}$ 
and let $\theta_0 \in \mathring{C}$ be a true parameter. 
Suppose that evaluation points $(u_1,\ldots,u_k)\in \R^k$ of $g$ satisfy $|u_i|\neq |u_j|$ for $i\neq j$ and $u_i \neq 0$ for all $i$.
Then the probability that $\wt S_n(\theta)=0$ has at least one root tends to $1$ as $n \to \infty$, 
and there exists a sequence of roots $\hat \theta_n$ such that $\hat \theta_n\to \theta_0$ in probability.  
Moreover, every consistent sequence of roots $\hat \theta_n$ satisfies  
\begin{align}
\label{asympt:linear:expression}
 \sqrt{n}\, (\hat \theta_n-\theta_0) = \wt I(\theta_0)^{-1}
 \frac{1}{\sqrt{n}} \sum_{j=1}^n \wt S(X_j;\theta_0)+o_p(1). 
\end{align}
In particular, the sequence $\sqrt{n}\,(\hat \theta_n-\theta_0)$ is
 asymptotically normal with mean zero and covariance matrix $\wt
 I(\theta_0)^{-1}$.
\end{theorem}

Theorem \ref{thm:asymptocs:fixedk} does not rule out the existence of multiple roots.
However, non-uniqueness of solutions is not a serious problem because we can utilize existing consistent estimators for $\theta_0$ such 
as the quantile-based estimator of McCulloch\cite{mcculloch:1986}.
We can find the consistent root by choosing the root closest to a preliminary consistent estimator (see p.70 of \cite{vandervaat:2000}).

The proof of Theorem \ref{thm:asymptocs:fixedk} is given in Appendix \ref{sec:append:pf:thm:asymptocs:fixedk}. 
Here we make a remark about the idea behind the proof.

\begin{remark}
The TML estimator could be regarded as a $Z$-type estimator whose estimating equation approximates 
the first-order condition of MLE.
If the estimator converges to a local minimum/maximum of the likelihood function, then consistency fails.
We can avoid this at least locally 
through the condition $[\, |u_i|\neq |u_j|\, \&\,  u_i \neq 0\,]$ $($see  
Lemma \ref{lem:auxiliary1}$)$, that is, the condition that for sufficiently large samples,
the Hessian 
of TMLE $($derivative of $-\wt S (\theta)$ w.r.t. $\theta)$ 
is positive definite in a neighborhood of true parameter $\theta_0$ 
$($cf. Remark \ref{rem:sec:triscore} $(\mathrm{i})$$)$.
Thus asymptotically and locally we can choose a unique zero which maximizes  
the likelihood function.  
\end{remark}

\subsection{Asymptotics of TMLE when the number of points goes to infinity}
\label{subsec:asymptQMLinfty}
We show that the TML estimator converges 
to the ML estimator in the sense that they have the same asymptotic law (Theorem \ref{thm:asympeffi}).
Throughout this subsection,
we take equally spaced evaluation points $(u_1,u_2,\ldots,u_k)=(\tau,2\tau,\ldots, k \tau)$ 
and study the limit behavior of TMLE when the number $k$ increases to 
infinity, while
the interval length $\tau>0$ goes to zero.

Referring to Brant \cite{brant:1984},
we first describe that TSF can be regarded as a two-step
approximation of the true score function. 
From this one can grasp how the limit operation $k\to \infty$ and $\tau\to 0$ works, which 
is also useful to follow the proofs.
The first step is based on a wrapping of $f(x;\theta)$, 
\begin{align}
\label{poissonsum}
 f_\tau(x;\theta) = \sum_{j=-\infty}^\infty f\,\Big(x+
\frac{2\pi j}{\tau};\theta
\Big),\quad \tau>0
\end{align}
from which we construct  
a wrapped version of the score function ($4\times 1$ vector) 
\begin{align}
\label{Stau}
 S_\tau (x;\theta) = \frac{\frac{\partial}{\partial
 \theta}f_\tau(x;\theta)}{f_\tau(x;\theta)} \left(=:
 \frac{f_{\tau,\theta}(x;\theta)}{f_\tau (x;\theta)} \right).
\end{align}
Notice that both $f_\tau$ and $S_\tau$ are periodic functions with period $2\pi/\tau$. 
By letting $\tau \to 0$ we obtain convergences $f_\tau \to f$ and $S_\tau \to S$ (Lemmas \ref{lem:tauapprox:derivatives} and \ref{lem:uniconti:s}).

The second approximation step is the projection of $S_\tau(x;\theta)$ into the 
subspace of $L^2(f_{\tau})$ that is spanned by trigonometric functions,
as done in Section \ref{sec:qscore}. 
Here, $L^2(f_{\tau})$ 
is the $L^2$ space on
$[-\pi/\tau,\,\pi/\tau]:=I_\tau$, 
whose inner product is given by 
\[
 \langle h_1,h_2 \rangle_{\tau,\theta} = \E_{\tau,\theta} [h_1\bar h_2]= 
\int_{I_\tau} h_1(x) \bar h_2(x)
 f_\tau(x;\theta) dx,\quad h_1,\,h_2\in L^2(f_\tau).  
\]
Now setting $(u_1,\ldots,u_k)=(\tau,2\tau,\ldots,k\tau)$ in the previous
definitions \eqref{def:gamma} and \eqref{def:deriv:gamma}, we obtain 
\[
 \E_{\tau,\theta} [(g(X)-\gamma(\theta))\, (g(X)-\gamma(\theta))^\top] 
 =\Sigma(\theta)\quad \text{and}\quad \E_{\tau,\theta} [S_\tau(X;\theta)\,(g(X)-\gamma(\theta))^\top ]
 =\gamma_\theta (\theta).
\]
Thus the projection coincides with TSF $\wt S$ with equidistant evaluation points.
Since $S_\tau$ is periodic, its projection on trigonometric functions will be close to $S_\tau$ itself for a sufficiently large $k$.
We can rigorously prove the convergence $\wt S\to S_\tau$ as $k\to \infty$ (Proposition \ref{prop:upperbounds}). 
Therefore, we obtain $\wt S \to S$ as $k \to \infty$ and $\tau \to 0$.  
\begin{remark}
Notice that TSF 
is originally obtained by the
projection of the true score function $f_\theta(x;\theta)/f(x;\theta)$ into 
the subspace of the $L^2(f)$ space spanned by $\{1,g(x)\}$. 
The equivalence of the two procedures 
follows from the same logic as that 
the Fourier coefficients of a wrapped function are 
equivalent to the Fourier transform of $f_\theta (x;\theta)$ at the corresponding
points $k\tau,\,k\in \N$ $($see also \cite[p.993]{brant:1984}$)$, that is, for $t=k \tau $,
\[
 \E_{\tau,\theta} [S_\tau(X;\theta)e^{itX}]=\int_{I_\tau} e^{itx} f_{\tau,\theta} (x;\theta) dx =\int_{-\infty}^\infty e^{itx} f_\theta(x;\theta) dx =\E_\theta [S (X;\theta) e^{itX}]
= \varphi_\theta(t;\theta). 
\]
In the proofs of this subsection, we implicitly use properties of Fourier series expansion of the
wrapped version $S_\tau(x;\theta)$ $($e.g. Theorems \ref{thm65:brant:1984} and \ref{thm66:brant:1984}$)$. 
Thus we take equidistant points on $g$ and consider 
$L^2(f_\tau)$. 
\end{remark}

It is not straight forward to extend Theorem \ref{thm:asymptocs:fixedk} to the case of $k \to \infty$ and $\tau \to 0$.
The difficulty stems from the fact that the proof of the theorem evaluates the second-order derivatives of $\wt S(x; \theta)$ with 
respect to $\theta$, which is quite complicated when $k \to \infty$ and $\tau \to 0$. Instead we apply \cite[Theorem 5.7]{vandervaat:2000}, 
the assertion of which is weaker than that of Theorem \ref{thm:asymptocs:fixedk} but assures that sequences $\hat \theta_n$ of 
(approximate) zeros of $\wt S_n(\theta)$ include a consistency sequence.

Let $\bar B_\delta(\theta_0)=\{\theta : |\theta - \theta_0| \le \delta \}$ be a closed ball with center $\theta_0$ and radius $\delta$.
We have the following theorem.

\begin{theorem}
\label{thm:consistency}
Let $C$ be any compact subset of $\mathring{\Theta}$ and $\theta_0\in \mathring{C}$ be a
 true parameter.
Then, for any $k$ and $\tau$, there exists a sufficiently small $\delta>0$ such that 
any sequence of estimators $\tilde \theta_n$ satisfying $\tilde \theta_n \in \bar B_\delta(\theta_0) \cap C$ for all sufficiently large $n$ and 
$|\wt S_n(\tilde \theta_n)|=o_p(1)$ converges in probability to $\theta_0$.  
 The result holds even when $\tau \to 0$ if $k^{-1}=o(\tau^{2+\alpha_0+\delta})$. 
\end{theorem}

The problem of TMLE is that $|\wt S_n(\theta)|$ potentially has multiple local minima.
However, $|\wt S_n(\theta)|$ asymptotically has a unique minimum in a neighborhood of $\theta_0$.
The theorem states that if we restrict the parameter space in a neighborhood of $\theta_0$, then any sequence of estimators $(\tilde{\theta}_n)$ that satisfies $|\wt S_n(\tilde{\theta}_n)|=o_p(1)$ is consistent for $\theta_0$.

Theorem \ref{thm:consistency} immediately implies the existence of a consistent sequence of TMLE.
Indeed, let $\tilde \theta_n = \arg \min_{\theta \in \bar B_\delta(\theta_0) \cap C} |\wt S_n(\theta)|$ for a sufficiently small $\delta >0$.
Then, from the proof of Theorem \ref{thm:consistency}, we see that $\tilde \theta_n$ satisfies $|\wt S_n(\tilde \theta_n)|=o_p(1)$, and thus that $\tilde \theta_n$ is consistent.
We emphasize that the introduction of $\bar{B}_\delta(\theta_0)$ is only for theoretical exposition.
By the existence of a preliminary consistent estimator, we can obtain a consistent sequence of TMLE without knowing $\bar{B}_\delta(\theta_0)$.

\begin{remark}
\label{remark:consistency:TMLE}
$(\mathrm{i})$ 
Similarly as in the fixed-points case, 
we can choose a sequence of consistent roots $\hat \theta_n$ of $\wt S_n(\theta)=0$.  
Practically, we use a variant of the Newton--Raphson algorithm with a consistent initial value. 
See Section \ref{sec:numerical} for the detail.
\\
$(\mathrm{ii})$ 
We observe in the proof that under the condition of $\tau$ and $k$ in Theorem \ref{thm:consistency}, 
$\wt S(x;\theta)$ converges to the true score $S(x;\theta)$ as $\tau \to 0$.  
By this fact, consistency of $\hat \theta_n$ is kept even in the limit $\tau \to 0$.  \\
$(\mathrm{iii})$ 
The condition $k^{-1}=o(\tau^{2+\alpha_0+\delta})$ of Theorem \ref{thm:consistency} depends on an unknown true parameter $\alpha_0$ and $\delta$. 
Because $\delta$ can be arbitrarily small, there exists $\delta$ such that $\alpha_0 + \delta \le 2$ for any $\theta_0 \in \mathring{C}$.
Then, the condition is satisfied if $k^{-1}=o(\tau^4)$ although using too many evaluation points is computationally demanding.
In actual implementation the dependence on $\alpha_0$ is not a major issue, see Remark \ref{rem:choice:tau:k}. 
Notice that the condition of $k$ and $\tau$ for consistency is included in 
that for asymptotic normality, see Theorem \ref{thm:asympeffi}. 
\end{remark}

Next we study the asymptotic normality and efficiency of the TML estimator, taking the consistent sequence $(\hat \theta_n)$ 
of TMLE. 
For this purpose, first, we evaluate the distance between the score of TML $\wt S(x;\theta)$ 
and that of MLE $S(x;\theta)$, which can be $o(n^{-1/2})$ depending on $k,\,\tau$  
(Lemma \ref{lemma:asymptoticnormal}). 
Then, we evaluate the distance between 
$\sqrt{n}(\hat \theta_n-\theta_0)$ and $\sqrt{n}(\tilde \theta_n-\theta_0)$, whose expressions are given by scores 
$\wt S$ and $S$ respectively, and show that it converges to $0$ in probability. 

\begin{theorem}
\label{thm:asympeffi} 
 Let $C$ be any compact set on $\mathring{\Theta}$ and $\theta_0\in \mathring{C}$ be a
 true parameter. 
 Denote the ML estimator and the consistent TML estimator 
 respectively by $\tilde \theta_n$ and $\hat \theta_n$. 
 If $k\to \infty$ and $\tau\to 0$ such that $n^{1/2}\big\{
 \tau^{\alpha_0} \log 1/\tau + (\tau^{2+\alpha_0+\delta}k)^{-1/2}
 \big\}\to 0$ holds for $\delta>0$ of Theorem \ref{thm:consistency}, then $\sqrt{n}\,(\hat \theta_n -\tilde
 \theta_n)=o_p(1)$; thus, $\hat \theta_n$ has the same asymptotic
 distribution as $\tilde \theta_n$. 
\end{theorem}

Theorem \ref{thm:asympeffi} states that the TML estimator is asymptotically normal and efficient because 
it inherits the asymptotic properties of the ML estimator \cite{matsui:2019}.
As far as we know, our result is the first to rigorously establish the asymptotic normality and efficiency 
of the estimator based on the countable points of the ch.f. 
Here we see the difference between our result and previous results. 

Previous studies show that the asymptotic variance of the GMM estimator, which is obtained under fixed $k$ and $\tau$, can be arbitrarily close to 
the Cr\'amer--Rao lower bound 
when $k$ is large and $\tau$ is small (\cite{feuerverger:mcDunnough:1981a, yu:2004}).
In other words, they employ a sequential asymptotic framework, 
where the limit w.r.t. $n$ is taken first and then the limit w.r.t. $k$ and $\tau$ is taken. 
However, the sequential asymptotic theory does not provide a good approximation to the finite sample distribution in general.
If the number of moment conditions is too large, the GMM estimator does not satisfy asymptotic normality (see, for instance, \cite{donald:imbens:newey:2003} 
and references therein).

Kunitomo and Owada \cite{kunitomo:owada:2006} showed that the EL estimator satisfies asymptotic normality 
if the number of moment conditions grows slowly depending on the sample size.
However, their asymptotic framework is also different from ours. Their proof consists of two steps.
First, they choose $\tau =K/k$ for fixed $K >0$, and prove that the EL estimator converges in distribution 
to a normal distribution if $k$ satisfies $k=O(n^{\frac{1}{3}-\eta})$ with $\eta >0$, so that $k$ is not independent of the sample size $n$. 
Then, they show that the asymptotic variance converges to the Cr\'amer--Rao lower bound as $K \to \infty$.
Therefore, they do not specify the condition for the relation between $k$ and $\tau$ under which the EL estimator 
is asymptotically normal and efficient. 

We close this section with a brief discussion of finite sample behaviors. 
As stated 
above 
the number of moment conditions is  
limited by sample size  
for both GMM and EL estimators. 
Since the number should be increased for efficiency, those estimators could not be optimal in finite samples. 
The CGMM estimator avoids this deficiency and can exploit a full continuum of moment conditions. 
However, similar to the GMM estimator, it has the estimation step of the optimal covariance operator. 
In view of simulations, this step causes some efficiency loss. Notice that 
the EL estimator skips this step since it estimates both the parameter vector and the optimal weight all at once. 
Taking above facts into consideration, we consider TMLE. In the TML estimator, 
$\tau$ and $k$ can be arbitrarily small and large respectively, regardless of the sample size. 
Moreover even if $k$ is sufficiently large, no singularity is observed, that is, it 
just converges to the ML estimator as $k\to\infty$ and $\tau\to 0$. 
Notice that in view of the form \eqref{def:TMLE}, the explicit weight matrix $\gamma_\theta(\theta)\Sigma(\theta)^{-1}$ includes $\theta$, 
and thus TMLE implicitly estimates the weight matrix simultaneously with $\theta_0$.  

\begin{remark}
\label{rem:choice:tau:k}
	The condition on $\tau$ and $k$ in Theorem \ref{thm:asympeffi} depends on the unknown parameter $\alpha_0$.
	In actual implementation, the dependence on $\alpha_0$ is not a major issue.
	Our simulation shows that a use of $100$ evaluation points is sufficient for efficient estimation for any value of $\alpha_0$.
	Using many more evaluation points does not improve the efficiency of the estimator when $n=1000$.
	This is probably because the estimation error of MLE dominates the approximation error of TSF when $n$ is not so large.
\end{remark}

\begin{remark}
\label{rem:asympt:expgmm}
The explicit GMM estimator is well defined even when we 
increase the evaluation number $k$ of ch.f. depending on $\tau$. 
In particular we could take the equally spaced points as done in the TML estimator. 
However, it remains to be seen whether the asymptotic efficiency beyond a sequential framework holds 
in the limit of $k$ and $\tau$. Indeed for the continuous version of GMM $($see \cite[Equation (5.1)]{feuerverger:mcDunnough:1981a}$)$, which could be regarded 
as the limit of $k$ and $\tau$, it is shown that the optimal weight is not always integrable $($\cite[Equation (5.9)]{feuerverger:mcDunnough:1981a}$)$, 
and indeed we could not obtain it. Although it is not explicitly mentioned in \cite{feuerverger:mcDunnough:1981a}, 
by ignoring the integrability and proceeding with some calculations, 
the first-order condition of explicit GMM with the optimal weight leads to the likelihood equation. 
We stay with only two certain facts: $(a)$ The score of explicit GMM is expressed by a combination of trigonometric functions and thus 
it has a more distance to the score of MLE than that of TMLE does, $(b)$ If we estimate 
$\hat \theta^{k+1}_n$ by the explicit GMM with $\Sigma (\hat \theta_n^k)$ iteratively it becomes close to 
the TML estimator. Indeed, in a small simulation study, which is not reported here, 
we could not find a clear difference in the mean square errors between the TML estimator and
 the explicit GMM using $\Sigma(\tilde \theta_n)$ with $\tilde \theta_n$ a consistent estimator. 
\end{remark}

\section{Simulation results}
\label{sec:numerical}
The performance of ch.f.-based estimators is examined in this section.
We compare the TML estimator with the EL estimator and the CGMM estimator. See Subsection \ref{subsec:comparison} for 
definitions and relations of these estimators. 
All of the estimators are asymptotically efficient under certain conditions. 

To find a root of $\wt S_n(\theta)=0$ for TMLE, we utilize the method of scoring algorithm, which is a variant of the Newton--Raphson algorithm. 
The method replaces the derivative of the objective function with its expected value in the iteration process; that is, 
given a preliminary estimator $\hat{\theta}_n^0$, the sequence $(\hat{\theta}_n^l)_{l\ge 1}$ is renewed by
 \[
 \hat{\theta}_n^{l+1} = \hat{\theta}_n^l + \delta  \wt I (\hat{\theta}_n^l)^{-1} \wt S_n(\hat{\theta}_n^l),
 \]
where 
$\delta$ ($0 < \delta \le 1$) is a scalar that adjusts the step length.
The  
procedure is computationally stable because the positive definiteness of $\wt I (\theta)$ is guaranteed by Lemma \ref{lem:auxiliary1}.
If the initial value of the iteration is $\sqrt{n}$-consistent and $\delta =1$, then the one-step estimator $\hat{\theta}_n^1$ is asymptotically equivalent to the TML estimator (\cite[Theorem 5.45]{vandervaat:2000}).
However, we 
iterate the procedure until the sequence converges. 

To obtain the TML and EL estimators, we need to determine a vector of grid points 
$u = (u_1, \dots, u_k)$ for $g$.
The TML estimator uses $101$ equidistant points $u=(0.01, 0.06, \dots, 5.01)$.
The EL estimator uses 11 points $u=(0.1, 0.6, \dots, 5.1)$ when $\alpha \neq 1.9$ and 6 points $u=(0.1, 1.1, \dots, 5.1)$ when $\alpha = 1.9$.
The selection of the grid points for the EL estimator is similar to that of \cite{kunitomo:owada:2006}.
We use only 6 points when $\alpha =1.9$ because our simulation showed that the EL estimator with 11 grid points performs poorly when $\alpha \ge 1.7$. 

For the minimization procedure for CGMM, 
because solving the original minimization problem of \cite{carrasco:florens:2000} is computationally involved, 
we utilize Proposition 3.4 of \cite{carrasco:chrnov:florens:ghysels:2007}, 
which gives a simpler yet equivalent minimization problem. To evaluate the integrals used in the objective function, we adopt the Gauss--Hermite formula.   
Recall that the problem of CGMM is that an estimate of $K$ is not invertible, and we need the regularization to estimate $K^{-1/2}$. 
The regularization parameter 
is set to be $0.01$ in our simulation.
Although we repeated the estimation using different values of the regularization parameter, 
the result was insensitive to this choice in this finite sample simulation. 

The simulation is conducted using R software with package {\it libstableR}. 
The data generating process takes six different values of $\alpha \in \{0.5, 0.7, 1.0, 1.3, 1.6, 1.9 \}$ and two different values of $\beta \in \{0, 0.5\}$.
The values of $\sigma$ and $\mu$ are fixed to $1$ and $0$, respectively.
As the initial value of all estimators,
we use the quantile-based estimator of McCulloch \cite{mcculloch:1986}, which is consistent and easy to calculate but not necessarily efficient.

Table \ref{comparison} reports the result of 1000 repetitions with 1000 observations. 
The mean, standard deviation, skewness, and kurtosis are reported.
We see that all estimators are nearly unbiased.
Moreover, except for a few cases, we do not see a clear deviation from a normal distribution in term of the skewness and kurtosis.
However, the standard deviation is quite different among estimators.
The TML estimator dominates other two estimators for all cases even though its computational burden is quite small.
The CGMM estimator performs poorly especially when the true value of $\alpha$ is small.
In contrast, the EL estimator performs rather poorly when the true value of $\alpha$ is large.

\begin{tiny}
	\begin{table}
		\begin{center}
				\caption{Simulation result}
				\label{comparison}
				\centering
				\begin{tabular}{|lll|cccc| cccc| cccc|}
				\hline
					& & &\multicolumn{4}{c|}{QMLE} & \multicolumn{4}{c|}{ELE} & \multicolumn{4}{c|}{CGMM} \\
					 \cline{4-15} 
					$\alpha$ & $\beta$ & & $\hat{\alpha}$ & $\hat{\beta}$ & $\hat{\sigma}$ & $\hat{\mu}$ & $ \hat{\alpha}$ & $\hat{\beta}$ & $\hat{\sigma}$ & $\hat{\mu}$ & $\hat{\alpha}$ & $\hat{\beta}$ & $\hat{\sigma}$ & $\hat{\mu}$ \\
					\hline
					0.5 & 0 & Mean & 0.5001 & -0.0023 & 1.0000 & -0.0009
					  & 0.5000 & -0.0029 & 0.9980 & 0.0008 & 0.5006 & -0.0027
					   & 0.9996 &  -0.0013
					\\
					&	 &	Sd & 0.0191	& 0.0455 & 0.0773 & 0.0266 & 0.0245  & 0.0582 & 0.0827 & 0.0305 & 0.0584 & 0.1286 & 0.1092 &  0.0444
					\\
					& & Skew & 0.1126 & -0.0615 & 0.4498 & -0.1202 & 0.1309  & 0.0033 & 0.3717 & 0.0900 & 0.0106 & 0.0222 & 0.0715 & 0.0665
					\\
					&  & Kur & 2.9610  & 3.0564 & 3.5118 & 2.8604 & 3.0749 & 3.1392  & 3.2826 & 3.1535 & 3.1382 & 3.1803 & 2.9802 & 3.0663
					\\
					 & 0.5 & Mean & 0.4995 & 0.4997 & 1.0038  & 0.0009 & 0.4980  & 0.5019 &1.0042 & 0.0011 & 0.5002 & 0.4996  & 1.0010  &  0.0021
					\\
					&	 &	Sd & 0.0181	& 0.0387 & 0.0752 & 0.0366 & 0.0232 & 0.0543 & 0.0778 & 0.0400 & 0.0584 & 0.1282 & 0.1094 &  0.0546
					\\
					& & Skew & 0.0128 & -0.0277 & 0.1693 &  0.2090 & 0.1800 & 0.0950 & 0.2312 & 0.2377 & 0.1184 &  0.1164 & -0.0054 &  0.0166
					\\
					&  & Kur & 2.8393 & 2.8909 & 3.0163 & 3.2009  & 2.7689 & 2.9204 & 3.0465 & 3.2012 & 3.0269 & 2.9172 & 2.8122 & 3.0560
					\\
					\hline
					0.7 & 0 & Mean & 0.7000	& -0.0016  & 1.0032 & 0.0002 & 0.7008 & -0.0039 & 0.9958 & 0.0018 & 0.7001 & -0.0071 & 0.9928 &  0.0020
					 \\
					 &	 &	Sd & 0.0240 & 0.0436 & 0.0575 & 0.0337 & 0.0301 & 0.0552 & 0.0609 & 0.0365 & 0.0598 & 0.1051 &  0.0785 &  0.0490
					 \\
					 & & Skew & 0.0809 & 0.0572 & 0.0975 & -0.0026 & 0.2372 & 0.0630 & 0.2889 & -0.0138 & 0.1167 & 0.0224 & 0.1172 & -0.0165  
					 \\
					 &  & Kur &  2.8056 & 2.8497 & 3.0761 & 2.9722 &  3.0582& 3.0206 & 3.0453 & 3.0562 & 2.9967 & 2.7643 & 2.88841 & 2.9925
					 \\
					 & 0.5 & Mean & 0.7023 & 0.4997 & 0.9987  & 0.0012 & 0.7017 & 0.4988 & 0.9993 & 0.0004 & 0.7035 & 0.5041 & 0.9969 &  0.0011
					 \\
					 &	&	Sd &0.0243  & 0.0399 & 0.0549  & 0.0409 & 0.0303 & 0.0500 & 0.0580 & 0.0442 & 0.0607 & 0.1042& 0.0723 & 0.0553 
					 \\
					 & & Skew & 0.2565 & -0.0634 & 0.3667 &  0.3090 &  0.1111 & -0.1013 & 0.1919 & 0.0458 & -0.0167 & 0.1035 & 0.1691 &  0.2004  
					 \\
					  & & Kur & 2.9837 & 2.8447 & 3.0964 & 3.2816 & 2.7998 & 2.8292 & 2.8247 & 2.8587 & 2.8587 & 3.2082 & 3.2045 & 3.3522 
					  \\
					  \hline
					1.0 &	0 &	Mean &1.0018 & -0.0004 & 1.0038 & 0.0006 & 1.0025 & -0.0004 & 1.0036 & 0.0002 & 1.0005 & 0.0009 & 1.0030 &  0.0007
					 \\
					 &  & Sd & 0.0353  &0.0568 & 0.0453 & 0.0492 & 0.0396 &	0.0639 & 0.0466 & 0.0492 & 0.0611 & 0.0936 & 0.0524 & 0.0547 
					 \\
					 & & Skew & 0.1197 & -0.0608 & 0.1635 & 0.1755 & 0.0158 & -0.0283 & 0.1766 & 0.1077 & -0.0337 & -0.0345 & 0.1477 &  0.1223
					 \\
					 & & Kur & 3.0653 & 3.3013 & 2.8804 & 2.9999 &  2.8650 & 2.9374 & 2.8817 & 3.0077 &  2.6764 & 3.1695 & 2.8678 & 2.8860 
					 \\
					 & 0.5 & Mean & 1.0003 & 0.4995 & 1.0002 & 0.0002 & 1.0025 & 0.5003 & 0.9997 & 0.0002 & 1.0009 & 0.5024 & 0.9991 & -0.0006
					 \\
				  	 & & Sd & 0.0345 & 0.0478  & 0.0434	& 0.0477 & 0.0384 & 0.0583 & 0.0446 & 0.0493  & 0.0603  & 0.0930 & 0.0502  & 0.0562
				  	  \\
				  	  && Skew & 0.0915 & -0.1710 & 0.0779 & 0.0943 & 0.0296 & -0.0994 & 0.1300 & 0.1252 & 0.0966 & 0.0847 & 0.1006 & 0.1239  
				  	  \\
				  	  && Kur & 2.8549 & 2.8708 & 2.8244 & 3.0870  & 2.9937 & 3.0937 & 2.8579 & 3.1045 & 2.9499 & 3.1468 & 2.8947 & 2.9264  
				  	  \\
				  	  \hline
		    	1.3 & 0 & Mean & 1.3011 & -0.0027 &	1.0000 & 0.0001	& 1.3015 &  -0.0040 & 1.0002 & 0.0009 & 1.3006 & -0.0031 & 0.9986 &	-0.0003
			    \\
			    & & Sd & 0.0449 & 0.0744 & 0.0372 & 0.0523 & 0.0508 & 0.0821  & 0.0380  & 0.0533 & 0.0619 & 0.1011 & 0.0399 &  0.0560
     			\\
	    		&& Skew & -0.0304 & 0.0089 & 0.0887 & -0.0614 & -0.0008 &  0.1513 & 0.0864 & -0.0081 & -0.1612 & 0.0679 &  0.0605 &  0.0072   
		    	\\
		    	& & Kur &  2.9152 & 2.9449 & 2.9022 & 3.0298 & 2.8310 & 3.2701 & 2.8996 & 2.9636 & 3.1164 & 2.9232 & 3.0110 & 2.9021  		   
    			\\
	    		& 0.5 & Mean & 1.3014 & 0.5002 & 0.9996 & 0.0003 &  1.3025 & 0.4983 & 1.0003 & 0.0015 & 1.3034 & 0.5012 & 0.9995 & 0.0015 
		    	\\
			    & & Sd & 0.0439 & 0.0652 & 0.0366 & 0.0527 & 0.0504 & 0.0754 & 0.0377 & 0.0542 &  0.0593 & 0.1087 & 0.0403 & 0.0587 
	    		\\
		    	&& Skew & 0.0393 & -0.0369 & 0.1030 & 0.0184 & -0.0080 & -0.0293 & 0.1208 & 0.0148 & 0.0620 & -0.0295 & 0.1192 &  0.0673 
		    	\\
		    	&& Kur & 2.9720 & 2.7660 & 2.8661 & 3.1141 & 3.1213 & 2.9781 & 2.9181 & 3.0627 & 2.9224 & 2.8581 & 2.8740 & 3.0668 
		    	\\
		    	\hline
				1.6 & 0 & Mean & 1.6004 & 0.0004 & 0.9990 & -0.0016 & 1.6052 &  0.0008 & 1.0001 & -0.0019 & 1.5995 & -0.0039 & 0.9987 & -0.0002  
				\\
				& & Sd & 0.0491 & 0.1137 & 0.0318 & 0.0560 & 0.0605 & 0.1345 & 0.0321 & 0.0586 & 0.0552 & 0.1479 & 0.0332 & 0.0613 
				\\
 	   		    && Skew & -0.0663 & 0.0843 & 0.1365 & -0.0923 & 0.5137 & -0.0980 & 0.0637 & -0.0865 & -0.1163 & 0.0493 & 0.1434 & -0.1285  
			    \\
				&& Kur & 2.8879 & 3.2288 & 2.8530 & 3.0718 & 4.9116 & 2.9819& 3.2748 & 3.0262 & 2.9319 & 3.1401 & 2.8545 & 3.1150 
				\\
				& 0.5 & Mean & 1.6006 & 0.5033 & 0.9986 & -0.0009 & 1.6023 & 0.5111 & 0.9983 & 0.0000 & 1.6009 & 0.5033 & 0.9986 & -0.0003
				\\
				& & Sd & 0.0485 & 0.1008 & 0.0315 & 0.0553 & 0.0608 & 0.1303 & 0.0334 & 0.0610 & 0.0565 & 0.1483 & 0.0333 & 0.0607 
			    \\
				&& Skew & 0.0094 & 0.0804 & 0.0905 & -0.0294 & 0.0124 & 0.4189 & 0.2196 & 0.0206 & -0.0258 & 0.0723 & 0.1685 & -0.0662 
				 \\
				&& Kur & 2.9767 & 2.9599 & 2.9079 & 3.0665 & 2.8251 & 3.3305 & 2.9167 & 3.3463 & 3.1096 & 3.3187 & 2.8926 & 3.2160 
				\\
				\hline
				1.9 & 0 & Mean & 1.9025 & 0.0160 & 0.9998 & -0.0018 &  1.9022 & 0.0041 &  0.9986 & -0.0012 & 1.8957 & 0.0052 & 0.9985 & -0.0020
				 \\
				& &	Sd & 0.0390 & 0.3478 & 0.0267 & 0.0531  & 0.0452 & 0.4375 & 0.0286 & 0.0611 & 0.0425 & 0.5299 & 0.0297 & 0.0650   
				\\
				&& Skew & -0.0739 & 0.0623 & -0.0295 & -0.0352  & -0.5660 & -0.0011 & 0.1676 & -0.0335 & -0.2156 & -0.0312 &  0.1440 & -0.1065
				\\
				&& Kur & 2.9590 & 3.9513 & 2.9083 & 3.0549 &  2.9949 & 3.1632 & 2.9147 & 3.0477  & 2.8248 & 2.4583 & 2.8669 & 2.9937  
				\\
				& 0.5  & Mean & 1.9012 & 0.5317 & 0.9994 & -0.0009 	& 1.8997 & 0.5259 & 0.9980 & 0.0026 & 1.8917 & 0.4707 & 0.9981 & -0.0012
				\\
				& &	Sd & 0.0376 & 0.3027 & 0.0266 & 0.0532 & 0.0451 & 0.3659 & 0.0282 & 0.0596 &  0.0423 & 0.4651 & 0.0296 & 0.0642
					 \\
			    && Skew & -0.1684 & -0.3778 & 0.0187 & -0.0131& -0.6621 & -0.6294 & 0.1596 & 0.0285  & -0.2116 & -0.8586 & 0.1379 & -0.0667
				   \\
			    && Kur & 2.8764 & 3.4267 & 3.0886 & 3.1041   & 3.3209 & 3.6659 & 2.9539 & 2.9691 & 2.8167 & 3.4851 & 2.8465 & 2.9645 
				 \\
					\hline
			\end{tabular}
		\end{center}
	\end{table}
\end{tiny}

\begin{small}
\begin{table}
	\begin{center}
		\caption{Efficiency}
		\label{efficiency}
		\centering
		\begin{tabular}{|ll|cccc|}
			\hline
			$\alpha$ & $\beta$  & $\hat{\alpha}$ & $\hat{\beta}$ & $\hat{\delta}$ & $\hat{\mu}$ \\
			\hline
			0.5 & 0 & 0.800 & 0.772 & 0.978 & 0.913 \\
			      & 0.5&  0.821 &  0.805 & 0.969 &  0.974  \\
			0.7 & 0 & 0.956  & 0.952  &  1.018 & 1.006 \\
			& 0.5 & 0.924 &   0.883  & 1.001 & 1.006 \\
			1.0 & 0 & 0.988 & 0.965  & 1.009  &  0.955 \\
			& 0.5 & 0.997 &  0.992 & 1.005  & 1.006  \\
			1.3 & 0 & 0.999 & 1.003 & 1.026  & 1.009 \\
			& 0.5 & 0.999 &  1.006  & 1.006  & 1.015 \\
			1.6 & 0 & 0.995  &  0.988 & 1.017  & 0.996  \\
			& 0.5  & 0.986 & 0.997  & 1.003  &  1.008 \\
			1.9 & 0 & 0.913 & 0.824 & 0.987  & 1.002 \\
			& 0.5  & 0.922 &  0.886  & 0.983 & 0.999 \\
			\hline
		\end{tabular}
	\end{center}
\end{table}
\end{small}

Table \ref{efficiency} reports the ratio of the standard deviation of the ML estimator to that of the TML estimator under the same setting as that for 
Table \ref{comparison}. Thus, larger values indicate higher efficiency of TMLE.
Because MLE is computationally demanding, we use the theoretical standard deviation calculated by Nolan \cite{nolan:2020}.
The simulation generally supports our theory. 
The TML estimator is on par with the ML estimator for $\alpha$ between 1.0 and 1.6. 
However, for a smaller $\alpha$ we observe a slight degradation in the performance.

\begin{remark}
	$(\mathrm{i})$ The performance of the CGMM estimator may be improved if the 
	regularization parameter is determined 
	in a date-driven way $($cf. \cite{carrasco:kotchoni:2017}$)$.\\
	$(\mathrm{ii})$ Although the result is not reported here, we also investigated the performance of the TML estimator by using 11 grid points.
	The TML and the EL estimators have similar means and standard deviations in that case.
	This is not surprising, because two estimators are first-order equivalent if the same grid points are used.
	However, in a finite sample, the EL estimator performs poorly if a large number of grid points is used.\\
	$(\mathrm{iii})$ If the parameter regions are close to the boundaries, such as $\alpha=2$ or $\beta=\pm 1$, some elements of 
	the Fisher information matrix are known to diverge to $\infty$ $($see \cite{dumouchel:1973,nagaev:shkolnik:1988,matsui:2005}$)$. 
	It will be interesting to see what happens for these estimators in such situations, and we will pursue this in future work. \\
	$(\mathrm{vi})$ In finite samples, the optimal choice of the grid points in the EL and TML estimators will depend on the true value of $\alpha$. 
	If $\alpha$ is small, the estimators with more grid points around the origin often perform better. Further study of 
	this topic will also be part of future work.
\end{remark}

\section{Extensions to estimation of stable OU-processes} 
\label{application:ou}
As an extension, we estimate the parameters of $\alpha$-stable
Ornstein--Uhlenbeck processes 
$(X_t)_{t\ge 0}$ based on 
$n$ discrete observations $(X_h,X_{2h},\ldots,X_{nh})$ 
with an interval $h>0$. By letting $n\to \infty$ with/without $h\to 0$, we can consider both high and low frequently observed 
processes. 
Let
$(Z^\alpha_t)_{t\in \R}$ be a symmetric $\alpha$-stable L\'evy process
with $\E [e^{iuZ_1^\alpha}]=e^{-|\sigma u|^\alpha},\,u\in \R$, where $\sigma>0$
is the scale parameter. Then a stationary version
$(X_t)_{t\ge 0}$ has the form 
\begin{align}
\label{def:sou}
 X_t= \int_{-\infty}^t e^{-\lambda (t-v)} dZ^\alpha_v=  \int_s^t e^{-\lambda(t-v)} dZ^\alpha_v + e^{-\lambda(t-s)} X_s, \quad t\ge 0.  
\end{align}
Notice that \eqref{def:sou} satisfies the stochastic differential equation 
\[
 dX_t = -\lambda X_t dt + dZ^\alpha_t
\]
and that due to the independent increments of $(Z^\alpha_t)$, the process is
a Markov process (see, e.g. \cite[Sec. 3.6]{samorodnitsky:taqqu:1994} for
the definition and properties of $\alpha$-stable
Ornstein--Uhlenbeck processes). 
For this type of Markov processes, several estimation methods based on 
the ch.f. have been established. 
One finds a detailed review in \cite[Sec. 2.2.2]{yu:2004} (see also 
\cite{singleton:2001,chacko:viceira:2003,carrasco:chrnov:florens:ghysels:2007}). All of them are
variants of the GMM estimator based on the conditional ch.f. 
Our estimator also relies on the conditional ch.f., but it appears as the result of 
projection of the conditional score function on trigonometric functions.

We see our estimator more closely. 
Write the conditional likelihood equation
\[
 \frac{1}{n} \sum_{t=1}^n S(X_{h(t+1)};\theta \mid X_{ht})  = \frac{1}{n} \sum_{t=1}^n
 \frac{f_\theta(X_{h(t+1)};\theta\mid X_{ht})}{f(X_{h(t+1)};\theta\mid X_{ht})} =0, 
\] 
where $\theta=(\alpha,\sigma,\lambda)^{\top}$ is our targeting parameter vector. Here
$f(x_{h(t+1)};\theta\mid x_{ht})$ and $f_\theta(x_{h(t+1)};\theta\mid x_{ht})$ are
respectively the conditional density of $X_{h(t+1)}$ given $X_{ht}$ and its
derivative vector. We approximate the likelihood equation by projecting 
the conditional score $f_\theta/f$ onto the subspace spanned by $\{1,g\}$. 
The estimator is obtained as a zero of this approximated equation. 
We
call this estimation trigonometrically approximated conditional MLE (TCMLE for short).  

We present the exact form of TCMLE. 
In the definition \eqref{def:sou} since the last two terms are independent, the conditional ch.f. is written
as 
\begin{align*}
 \varphi_{t\mid s} (u;\theta) := \E [e^{i u X_t} \mid X_s] 
 = e^{i u e^{-\lambda (t-s)} X_s} e^{-|\sigma u |^\alpha
 (\lambda \alpha)^{-1}(1-e^{-\alpha\lambda (t-s)})}. 
\end{align*}
A conditional version of $\gamma(\theta)$ at $t$ given $s$,\,$t>s$ is
provided by 
\[
 \gamma_{t\mid s}(\theta ) = \big(
\varphi_{t\mid s}^R(u_1;\theta),\ldots,\varphi_{t \mid s}^R (u_k;\theta ), 
\varphi_{t\mid s}^I (u_1;\theta),\ldots,\varphi_{t\mid s}^I (u_k;\theta )
\big)^\top, 
\]
where $(u_1,\ldots,u_k)$ is a vector of evaluation points of the ch.f. 
Accordingly we have 
\[
 \gamma_{t\mid s,\theta}(\theta)= \frac{\partial \gamma_{t\mid s} (\theta)}{\partial \theta} = \big(
 \varphi^R_{t\mid s,\theta}(u_1;\theta),\ldots,
 \varphi^R_{t\mid s,\theta}(u_k;\theta), 
 \varphi^I_{t\mid s,\theta}(u_1;\theta),\ldots,
 \varphi^I_{t\mid s,\theta}(u_k;\theta)
\big), 
\]
where 
\begin{align*}
 \varphi_{t\mid s;\theta}^R(u;\theta) = \frac{\partial
 \varphi_{t\mid s}^R(u;\theta)}{\partial \theta},\qquad 
  \varphi_{t\mid s;\theta}^I(u;\theta) = \frac{\partial
 \varphi_{t\mid s}^I(u;\theta)}{\partial \theta}.  
\end{align*}
The quantity corresponding to $\Sigma(\theta)$ is
 defined by 
\[
 \Sigma_{t\mid s}(\theta) = \E_{t\mid s}[(g(X)-\gamma_{t\mid s}(\theta))
 (g(X)-\gamma_{t\mid s}(\theta))^\top], 
\]
where $\E_{t\mid s}[\cdot]$ denotes the expectation 
operator deduced from the conditional density $f(x_t;\theta \mid x_s)$. 
Now we obtain the trigonometrically approximated conditional score function as  
\[
\widetilde S(x_{ht};\theta \mid x_{hs}) 
:= \gamma_{ht\mid hs,\theta}(\theta)\Sigma_{ht \mid hs}(\theta)^{-1}\big(
g(x_{ht})-\gamma_{ht\mid hs}(\theta) 
\big),
\]
so that the corresponding empirical score is given by 
\[
 \wt S_{n}(\theta) = \frac{1}{n} \sum_{t=1}^n 
\wt S(X_{h(t+1)};\theta \mid X_{ht}).  
\]
The TCML estimator is the solution $\hat \theta_n $ of $\wt S_{n}(\theta)=0$. 

\begin{remark}
 Notice that similarly to the i.i.d. case,  
 GMM-type estimators again encounter the singularity problem 
 in the covariance matrix \eqref{cov} 
 as the number of the grid points increases, see \cite{carrasco:florens:2000} $($cf. \cite[Sec. 2.2.2]{yu:2004}$)$. 
 Thus one needs to increase the grid number depending on sample sizes or needs a regularization parameter for the optimal covariance 
 operator. 
 Although we do not have a rigorous theoretical result for the non-i.i.d. case, in view of numerical simulations, we do not observe 
 such a singularity in TCMLE.
\end{remark}
Hereafter, we will examine the finite sample performance of our estimator in high frequency settings. 
In the literature asymptotic theories for quasi-MLE have been established 
for OU processes driven by more general locally stable L\'evy processes, 
especially for $\lambda$ \cite{masuda:2019,clement:gloter:2019,clement:gloter:2020}. 
The main tool is locally asymptotic stable property of the processes. 
The joint asymptotic distribution of all parameters 
$\theta$ has been derived by \cite[Theorem 3.2]{clement:gloter:2020}.
Since TMLE converges to MLE as $k\to\infty,\,\tau\to 0$, for reasonably large $k$ and small $\tau$ TCMLE should approximate this asymptotic distribution also.  
We briefly explain this asymptotic in order to compare it with our simulation results.

Assume that we observe $(X_{t})_{t\ge 0}$ on $[0,T]$ with frequency $h=T/n$, that is, we observe 
a discretized process $(X_{t_i})_{0\le i \le n}$ with $t_i=i\cdot T/n$.  
Let $a_n={\rm diag} (n^{1/2}\log n,n^{1/2},n^{1/\alpha_0-1/2})$ and then as $n\to \infty$ 
\[
 a_n (\hat \theta_n-\theta_0) \stackrel{d}{\to} V_0^{-1/2} N, 
\]
where $N$ is the standard normal r.v. independent of $V_0$ and 
\beao
V_0=
\left(\barr{ccc} c_{11} & c_{12} & 0 \\
c_{21} & c_{22} & 0 \\
0 & 0 & c_{33} \int_0^T X^2_u du 
\earr\right).
\eeao
Here $c_{ij},\,i,j=1,2,3$ are some constants depending on $T$. In other words, $\hat \theta_n$ has asymptotic mixed normality. 
If numerical simulations replicate this asymptotic mixed normality, 
then the validity of TCMLE would be assured.

\begin{table}[htb] 
\small
\begin{center}
\caption{TCMLE for $\alpha$-stable OU processes with frequency $h=0.1,0.01$}
\begin{tabular}{|l|cccc|cccc|cccc|}\hline
    $h=10^{-1}$ & \multicolumn{4}{|c|}{$\alpha=1.3$} & \multicolumn{4}{|c|}{$\alpha=1.5$} & \multicolumn{4}{|c|}{$\alpha=1.8$}\\ \hline
    $\hat \theta$   & $\hat \alpha$ & $\hat \sigma$& $\hat \lambda$ & $\lambda^\ast$ 
    & $\hat \alpha$ & $\hat \sigma$ & $\hat \lambda$ & $\lambda^\ast$ 
    & $\hat \alpha$ & $\hat \sigma$ & $\hat \lambda$ & $\lambda^\ast$ \\ \hline
    Mean & $1.299$ & $1.004$ & $1.001$ & $-.0748$ & $1.500$ & $1.003$ & $1.004$ &  $-.102$  & $1.801$  & $1.001$ & $1.013$ & $-.110$  \\
    Sd   & $0.050$ & $0.064$ & $0.036$ & $0.914$ & $0.050$ & $0.052$ & $0.061$ &  $1.128$  & $0.044$  & $0.036$ & $0.112$ & $1.384$  \\
    Skew & $0.073$ & $0.221$ & $0.284$ & $0.169$ & $0.044$ & $0.147$ & $0.336$ &  $0.074$  & $-0.126$ & $0.220$ & $0.383$ & $0.000$  \\
    Kur  & $2.626$ & $3.061$ & $5.407$ & $3.296$ & $2.682$ & $3.083$ & $3.875$ &  $3.429$  & $2.575$  & $3.086$ & $3.259$ & $2.971$  \\ \hline 
    Mean & $1.302$ & $1.001$ & $0.999$ & $0.008$ & $1.503$ & $1.000$ & $1.001$ & $-.043$ & $1.803$  & $1.000$ & $1.007$ & $-.078$  \\
    Sd   & $0.042$ & $0.053$ & $0.025$ & $0.922$ & $0.042$ & $0.042$ & $0.047$ & $1.139$ & $0.037$  & $0.028$ & $0.090$ & $1.405$  \\
    Skew & $0.135$ & $0.117$ & $0.183$ & $-.008$ & $0.003$ & $0.060$ & $0.336$ & $0.077$ & $-.125$  & $0.211$ & $0.338$ & $0.025$  \\
    Kur  & $2.753$ & $2.764$ & $4.504$ & $3.230$ & $2.764$ & $2.764$ & $4.592$ & $3.136$ & $2.830$  & $3.023$ & $3.606$ & $3.232$  \\ \hline 
    $h=10^{-2}$ & \multicolumn{4}{|c|}{$\alpha=1.3$} & \multicolumn{4}{|c|}{$\alpha=1.5$} & \multicolumn{4}{|c|}{$\alpha=1.8$}\\ \hline
    $\hat \theta$   & $\hat \alpha$ & $\hat \sigma$& $\hat \lambda$ & $\lambda^\ast$ 
    & $\hat \alpha$ & $\hat \sigma$ & $\hat \lambda$ & $\lambda^\ast$ 
    & $\hat \alpha$ & $\hat \sigma$ & $\hat \lambda$ & $\lambda^\ast$ \\ \hline
    Mean & $1.307$ & $1.020$ & $1.033$ & $-.206$ & $1.505$ & $1.012$ & $1.061$ &  $-.022$  & $1.800$  & $1.005$ & $1.123$ & $-.284$  \\
    Sd   & $0.129$ & $0.194$ & $0.182$ & $0.789$ & $0.110$ & $0.143$ & $0.260$ &  $0.094$  & $0.070$  & $0.078$ & $0.389$ & $1.275$  \\
    Skew & $0.089$ & $0.737$ & $1.703$ & $-.382$ & $-.064$ & $0.646$ & $1.203$ &  $-.021$  & $-.360$  & $0.602$ & $0.901$ & $-.357$  \\
    Kur  & $2.703$ & $3.832$ & $10.67$ & $4.415$ & $2.792$ & $3.544$ & $5.925$ &  $2.974$  & $2.795$  & $3.298$ & $4.218$ & $3.427$  \\ \hline 
    Mean & $1.309$ & $1.006$ & $1.010$ & $-.089$  & $1.506$ & $1.004$ & $1.025$ & $-.133$ & $1.802$  & $1.002$ & $1.074$ & $-.214$  \\
    Sd   & $0.107$ & $0.156$ & $0.109$ & $0.0683$ & $0.093$ & $0.118$ & $0.170$ & $0.851$ & $0.057$  & $0.063$ & $0.295$ & $1.195$  \\
    Skew & $0.166$ & $0.428$ & $0.453$ & $-.101$  & $-.035$ & $0.469$ & $0.814$ & $-.017$ & $-.376$  & $0.537$ & $1.057$ & $-.072$  \\
    Kur  & $2.844$ & $2.910$ & $4.889$ & $3.355$  & $2.689$ & $3.133$ & $5.172$ & $3.219$ & $2.871$  & $3.442$ & $5.089$ & $3.150$  \\ \hline 
  \end{tabular}\label{process:simula}
\vspace{2mm}\\
\begin{minipage}{15cm}
Monte Carlo simulations of TCMLE for parameters $(\alpha,\sigma,\lambda)$ of $\alpha$-stable OU processes with 
frequency $h=0.1$ (top two rows) and $h=0.01$ (bottom two rows). For the first and third rows we take $n=1000$ and 
for the second and fourth rows $n=1500$.
The estimation procedures are repeated $500$ times. Here $\lambda^\ast$ is a modified version of $\hat \lambda$, see
 \eqref{def:modlambda} for definition. In each cell $4$ statistical characteristics of estimators are presented.
\end{minipage}
\end{center}
\end{table}
In our simulation, we set $h=0.1,0.01$ and $n=1000,1500$ so that  
a path of process $(X_t)_{t\ge 0}$ on $t\in [0,nh]$ is observed in $n$ different times.  
We repeated the estimation procedure 
$500$ times in three different cases $\alpha=1.3, 1.5$ and $1.8$ with $\sigma=\lambda=1.0$. 
We uses $101$ equidistant points $u=(0.05, \dots, 5.05)$ for evaluation of 
empirical conditional ch.f. 
\begin{figure}[htbp]
\begin{center}
\includegraphics[width=\textwidth]{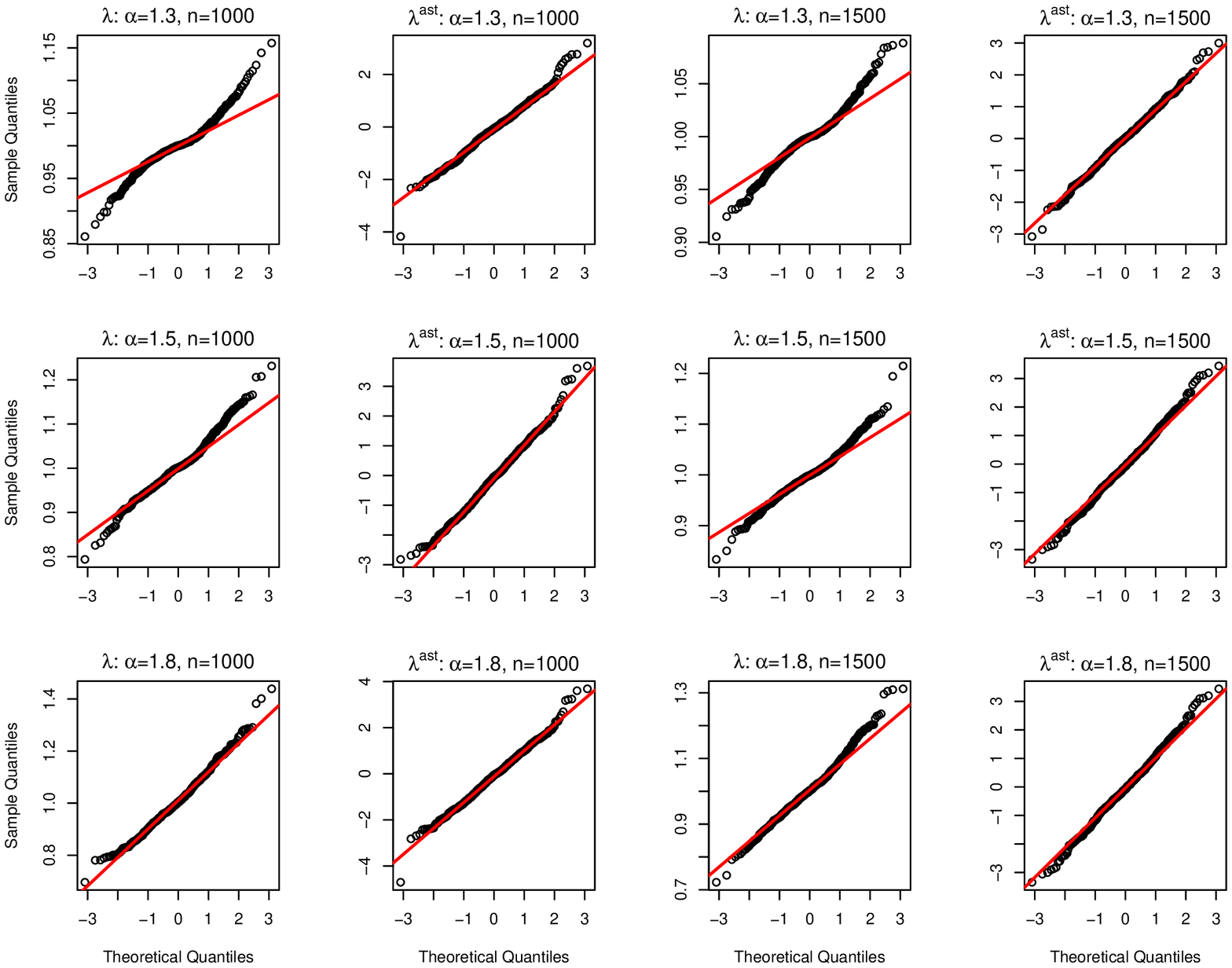}
\end{center}
\caption{
QQ plots for estimators $\hat\lambda_n $ 
and its modification $\lambda^\ast_n$ 
with sample sizes $n=1000$ (left two columns) and $n=1500$ (right tow columns), 
where outliers for $\lambda^\ast_n$ are removed. From top to bottom we change parameter 
$\alpha=1.3,1.5,1.8$ successively.}\label{fig:1}
\end{figure}
The results are summarized in Table \ref{process:simula}, where   
$4$ statistics are presented similarly as before. 
Judging from the mean values, all estimators are unbiased, 
while obviously the standard deviations of $n=1500$ are smaller than $n=1000$.
The standard deviations for $h=0.1$ are smaller than those for $h=0.01$. 
In view of skewness and kurtosis, asymptotic normality follows for $\hat \alpha_n$ and $\hat \sigma_n$ 
(We check QQ-plots also, though we do not present here), though these values for $\hat \lambda_n$ 
show quite different values. Values of $\hat \lambda_n$ are studied in detail bellow. 
We find difficulty in confirming orders $a_n$ directly, since accuracy of the asymptotic scale 
matrix $V_0$ depends on both $h$ and $T=nh$. Alternatively by observing the 
asymptotic mixed normality for $\hat \lambda_n$ we see that TCMLE properly functions.

Thus, we consider skewness, kurtosis, and QQ-plots for a version 
\begin{align}
\label{def:modlambda}
  \lambda^\ast_n:=\sqrt{W} (\hat \lambda_n-\bar \lambda),\quad W= \int_0^{nh} X^2_v dv,
\end{align}
where $\bar \lambda$ is the mean of $500$ samples of $\hat \lambda_n$. 
The values of $\lambda_n^\ast$ are presented in Table \ref{process:simula}. 
Since the original $\lambda^\ast_n$ sometimes takes huge outliers (mostly negative), 
we present further modified values of $\lambda^\ast_n$, that is, we remove several very small values. 
In cases $\alpha=1.3,n=1000,1500,h=0.01$ and $\alpha=1.5,n=1500,h=0.01$ we remove $1$ outlier, 
and in cases $\alpha=1.3,1.5,h=0.01,n=1000$, we remove $5$ outliers. 
Notice that skewness and kurtosis are sensitive to outliers, and the two characteristics of $\lambda^\ast_n$ often show worse values than those of 
$\hat \lambda_n$; for instance, in $\alpha=1.3,1.5,h=0.01,n=1000$, and then $($skewness, kurtosis$)$ are 
$(-1,43,11.44)$ and $(-3.50,31.37)$ respectively.

In Figure \ref{fig:1}, we present QQ-plots of $\hat \lambda_n$ and $\lambda^\ast_n$ for 
$\alpha=1.3,1.5,1.8$ and sample sizes $n=1000,1500$. Similarly to Table \ref{process:simula}, outliers of $\lambda^\ast_n$ are removed. 
In view of these graphs, $\hat \lambda_n$ is well-rescaled, and 
plots for $\lambda^\ast_n$ almost show normality. Judging from these facts we conclude that 
the theory of asymptotic mixed normality for $\hat \lambda_n$ is depicted correctly with TCMLE.


\appendix 
\section{Supplementary results and proofs for Section \ref{subsec:asymptotic}}
\label{append:suppl}

\subsection{Proof of Lemma \ref{lem:auxiliary1}}
\label{sebsec:a1}
$(\mathrm{i})$ 
It suffices to take $k=2$ with non-zero $|u_1|\neq |u_2|$ in \eqref{def:deriv:gamma} and prove that
rows of $\gamma_\theta(\theta)$
are linearly independent. 
We take a similar approach as in the proof of \cite[Theorem 3.1]{matsui:2019}, 
where the non-singularity of the Fisher information matrix is shown for the stable law.
Let $\bfa =(a_1,\ldots,a_4)^\top\in \R^4$ and we
show that 
\begin{align}
\label{pf:indirecthipo}
 \gamma_\theta(\theta)^\top \bfa =\bf0
\end{align}
implies $\bfa=\bf0$ by indirect proof. Assume \eqref{pf:indirecthipo} for
$\bfa\neq \bf0$, then 
\[
 \big(
 \varphi^R_\theta (u_1;\theta)+ i \varphi^I_\theta (u_1;\theta)
 , 
 \varphi^R_\theta (u_2;\theta) + i \varphi^I_\theta (u_2;\theta)
\big)^\top \bfa = 
 \big(
 \varphi_\theta (u_1;\theta), \varphi_\theta (u_2;\theta) 
\big)^\top \bfa = \bf0 
\]
holds. 
Therefore, the complex vectors 
\[
  \left(
    \begin{array}{c}
      \varphi_\mu(u_1;\theta) \\
      \varphi_\mu(u_2;\theta)
    \end{array}
  \right),   \left(
    \begin{array}{c}
      \varphi_\sigma(u_1;\theta) \\
      \varphi_\sigma(u_2;\theta)
    \end{array}
  \right),  \left(
    \begin{array}{c}
      \varphi_\alpha(u_1;\theta) \\
      \varphi_\alpha(u_2;\theta)
    \end{array}
  \right),
  \left(
    \begin{array}{c}
      \varphi_\beta(u_1;\theta) \\
      \varphi_\beta(u_2;\theta)
    \end{array}
  \right)
\]
are linearly  
dependent with real coefficients $\bfa\neq \bf0$.
For convenience we omit the parameter $\theta$ in $\varphi(u;\theta)$ and $\varphi_{\theta_i}(u;\theta)$ and functions of them.

For $\psi(u)=\log \varphi(u)$, we define $\psi_{\theta_i}(u)=\partial
\psi(u)/\partial \theta_i,\,i=1,\ldots,4$, so that we have
$\varphi_{\theta_i}(u)=\psi_{\theta_i}(u)\varphi(u)$, where 
\begin{align*}
\psi_\mu(u) &=
iu, \\
 \psi_\sigma(u) &= 
-\alpha
 |u|^\alpha +iu (\alpha|u|^{\alpha-1}-1)\beta \tan (\pi\alpha/2), \\
 \psi_\alpha(u) &= 
 -|u|^\alpha \log |u| +i u |u|^{\alpha-1} \log |u| \beta
 \tan (\pi\alpha/2) 
 + i u (|u|^{\alpha-1}-1)(\pi\beta/2) \cos^{-2}(\pi\alpha/2),\\
 \psi_\beta(u) &= 
 i u (|u|^{\alpha-1}-1)\tan (\pi\alpha/2). 
\end{align*}
Now we look $\sum_{i=1}^4 a_i\varphi_{\theta_i}(u_k)= \sum_{i=1}^4
a_i\psi_{\theta_i}(u_k) \varphi(u_k)$ 
for $k=1,2$. If these sums are zero,
then 
\[
 \sum_{i=1}^4 a_i \psi_{\theta_i}(u_k)=0. 
\]
In view of the real parts of $\psi_\sigma$ and $\psi_\alpha$, 
\[
 a_2\alpha+ a_3 \log |u_k|=0, \quad k=1,2,
\]
should hold for $|u_1|\neq |u_2|$. This is impossible unless
$a_2=a_3=0$. Moreover, from $\psi_\mu$ and $\psi_\beta$, we have 
\[
 a_1-a_4 \tan (\pi\alpha/2) + a_4
 \tan(\pi\alpha/2)|u_k|^{\alpha-1}=0,\quad k=1,2,
\]
which is impossible unless $a_1=a_4=0$. 
Notice that at $\alpha=1$, 
the above equation reduces to $a_1=a_4(2/\pi) \log |u_k|$, and thereby the result is kept. 
Thus, we obtain $\bfa=\bf0$, which is a
contradiction.  \\
\noindent
$(\mathrm{ii})$ Let $\bfa=(a_1,\ldots,a_{2k})^\top\neq \bf0$. Since
trigonometric functions in $g(x)$ constitute a basis on $L^2(-\pi,\pi)$
(see e.g. \cite[Theorem 1]{Christensen:2006}), they are linearly
independent. 
Thus, $\bfa^\top g(x)=0$ for any $x\in (-\pi,\pi)$ if and
only if $\bfa=\bf0$, cf. \cite[p.9]{kolmogorov:formin:2012}. This implies
that $\bfa^\top (g(x)-\gamma(\theta))$ is non-deterministic as a function
of $x$ for $\bfa\neq \bf0$. Now for any $\bfa\neq \bf0$ it follows from Fubini  
that 
\[
 \bfa^\top \Sigma(\theta) \bfa = \int_{-\infty}^\infty \{ \bfa^\top ( g(x)-\gamma(\theta))\}^2
 f(x;\theta) dx >0,  
\] 
since the support of $f(x;\theta)$ is the whole real line and $\bfa^\top
g(x)$ is
continuous in $x$. Thus, $\Sigma(\theta)$ is positive definite.\hfill $\square$ 
\begin{remark}
\label{rem:auxiliary1}
 $(\mathrm{i})$ We could not remove the condition $|u_i|\neq |u_j|$. Indeed, we take 
 $(u_1,u_2,\ldots,u_{k}),\,k\ge 3$ such that $u_1=-u_2$ and $u_i=0,\,i\ge 3$, and $\bfa=(1,-1,0,\ldots,0)$, and then 
$\bfa^\top g(x)=0$ holds for all $x\in \R$. \\
 $(\mathrm{ii})$ In view of \cite[Theorem 1]{Christensen:2006}, we do not necessarily use the same set of points $(u_1,u_2,\ldots,u_{k})$ 
for both $\varphi^R$ and $\varphi^I$. We might obtain a better estimator with two different sets of evaluation points.  
\end{remark}

\subsection{Proof of Theorem \ref{thm:asymptocs:fixedk}}
\label{sec:append:pf:thm:asymptocs:fixedk}
We will check the conditions of \cite[Theorem 5.41]{vandervaat:2000} where $\psi_\theta(x)$ 
there corresponds to $\wt S(x;\theta)$ here. 
Since the conditions of \cite[Theorems 5.41 and 5.42]{vandervaat:2000} are the same, the existence of a consistent root  
follows from Theorem 5.42, while asymptotic normality and \eqref{asympt:linear:expression} directly follow from Theorem 5.41. 

In view of \eqref{def:TMLE}, $\wt S(x;\theta)$ is twice continuously differentiable for every $x$ 
because the ch.f. $\varphi(x;\theta)$ is three times continuously differentiable w.r.t. $\theta$ on $C\subset\mathring{\Theta}$. 
Note that in \cite{matsui:2019}, the derivatives of $\varphi(u;\theta)$ w.r.t. $\theta$ are obtained up to the second order.
In a similar manner we could obtain the third-order derivatives even in the complicated case of $\alpha=1$. 

From \eqref{def:TMLE} and \eqref{hessian:tmle}, $\E_{\theta_0}[\wt S(X;\theta_0)]=0$ and $\E_{\theta_0}[|\wt S(X;\theta_0)|^2]<\infty$ follow. 
Write the first-order derivative matrix ($4\times 4$ denominator layout) as 
\[
 \frac{\partial \wt S(x;\theta)}{\partial \theta} = \Big(
\frac{\partial \wt S}{\partial \mu}^\top,\,\frac{\partial \wt S}{\partial
 \sigma}^\top, \,\frac{\partial \wt S}{\partial \alpha}^\top, \,\frac{\partial \wt S}{\partial \beta}^\top
\Big)^\top,
\]
where each row $\partial \wt S/\partial \theta_j,\,j=1,\ldots,4$ is
 a $ 1 \times 4$ vector and the $i,j$ element of $\partial \wt S/\partial
 \theta $ is 
 \begin{align}
\label{def:firstderiv}
  \Big[ \frac{\partial \wt S(x;\theta)}{\partial \theta } \Big]_{ij}
  & = \Big[
 \frac{\partial^2 \gamma(\theta) }{\partial \theta_j \partial \theta_i} +
  \frac{\partial \gamma(\theta) }{\partial \theta_i} \Sigma(\theta)^{-1}
  \frac{\partial \Sigma(\theta)}{\partial \theta_j}
 \Big] \Sigma(\theta)^{-1} (g(x)-\gamma(\theta)) \\
 &\qquad 
+ \frac{\partial \gamma(\theta) }{\partial \theta_j}
  \Sigma(\theta)^{-1}  \frac{\partial \gamma(\theta)}{\partial \theta_i}^\top. \nonumber 
 \end{align}
Thus $\E[\partial \wt S(X;\theta_0)/\partial \theta]=\gamma_\theta(\theta_0)\Sigma(\theta_0)^{-1}\gamma_\theta(\theta_0)^\top$ exists and is non-singular 
according to Lemma \ref{lem:auxiliary1} (cf. Remark \ref{rem:sec:triscore}).  

From the form of the first-order partial derivative \eqref{def:firstderiv}, it is not difficult to observe 
\begin{align}
\label{form:third:deriv}
 \frac{\partial^2 \wt S(x;\theta)}{\partial \theta \partial \theta_k}= A(\theta)(g(x)-\gamma(\theta)) + B(\theta),\quad k=1,2,3,4,
\end{align}
where $A(\theta)$ and $B(\theta)$ are continuous matrix functions in $\theta \in C$. 
Since $C$ is compact and $g$ is integrable, we could find a dominant integrable function $\ddot{\psi}(x)$ for each element of 
\eqref{form:third:deriv}.
Now due to \cite[Theorem 5.42]{vandervaat:2000}, there exists a consistent sequence of roots $\hat \theta_n$, and by 
\cite[Theorem 5.41]{vandervaat:2000} this sequence satisfies asymptotic normality together with \eqref{asympt:linear:expression}. \hfill $\square$ 

\section{Supplementary results and proofs for Section \ref{subsec:asymptQMLinfty}}
\label{sec:auxproofs}
For the proofs of Theorem \ref{thm:consistency} (Subsection \ref{append:thm:consitency}) and Theorem \ref{thm:asympeffi} 
(Subsection \ref{subsec:lemma:asymptoticnormal}), we need supplementary results. 
We use two theorems from \cite{brant:1984} (Theorems \ref{thm65:brant:1984} and \ref{thm66:brant:1984}), and the following lemmas and a proposition. 
Throughout this section, let $c,\,c',\,c''$, etc. be 
generic positive constants 
whose values are not of interest. 
For a given function $g(x)$, we denote by $g_\tau(x)$ the $2\pi/\tau$ periodic function:
\[
g_\tau(x) = \sum_{j=-\infty}^\infty g\, \big( x + 2\pi j/\tau \big),\quad \tau>0. 
\]
The first lemma is an easy application of Lemma 6.2 in \cite{brant:1984}. 
\begin{lemma}
\label{lem:tauapprox:derivatives}
 Let $f_\tau,\,f_\tau',\,f_{\tau,\theta_i}$ and $f_{\tau,\theta_i}',\,i=1,\ldots,4$ be periodic functions respectively based on $f,\,f',\,f_{\theta_i}$ and $f_{\theta_i}'$. Then, for any $\theta \in C$, we have 
\begin{align*}
 \sup_{x\in I_\tau} |f_{\tau}(x;\theta)-f(x;\theta)| &\le c_1 
 \tau^{1+\alpha},\qquad \sup_{x\in I_\tau} |f_{\tau}'(x;\theta)-f'(x;\theta)| \le c_2 
 \tau^{2+\alpha},\\
  \sup_{x\in I_\tau} |f_{\tau,\theta_i}(x;\theta)-f_{\theta_i}(x;\theta)| &\le c_3 
 \tau^{1+\alpha}\log 1/\tau, \qquad
  \sup_{x\in I_\tau} |f_{\tau,\theta_i}'(x;\theta)-f_{\theta_i}'(x;\theta)| \le c_4 
 \tau^{2+\alpha}\log 1/\tau,
\end{align*}
where $c_i,\,i=1,\ldots,4$ are positive constants that depend on neither $\theta$ nor $x$. 
\end{lemma}
\begin{proof}
Since the proof is a slight modification of that for \cite{brant:1984},
 we only state the difference in the condition. 
For sufficiently large $|x|$, we clearly have
$f=O(|x|^{-(1+\alpha)})$. 
In view of Lemma 2.1 in \cite{matsui:2019}, $f,\,f',\,f_{\theta_i}$ and
 $f_{\theta_i}'$ are continuous in $x\in\R$ and satisfy 
\begin{align*}
  |f'| = O(|x|^{-(2+\alpha)}) ,\quad
 |f_{\theta_i}|=O(|x|^{-(1+\alpha)}\log |x|),\quad
 |f_{\theta_i}'|=O(|x|^{-(2+\alpha)}\log |x|), 
\end{align*}
for sufficiently large $|x|$, where the orders of left are uniform in $i$. 
We replace the tail condition (6.1) of \cite[Lemma 6.2]{brant:1984} 
with the above bounds, and then the results follow from the reproduction of the proof.  
\end{proof}
As stated before Theorem \ref{thm:asympeffi}, we evaluate the distance between $\wt S$ and $S$ 
through the two-step approximation; 
we will evaluate the distance between $S$ and $S_\tau$ in Lemma \ref{lem:uniconti:s} and 
that between $S_\tau$ and $\wt S$ 
in Proposition \ref{prop:upperbounds}. 
In this regard we define several notations.
Denote the $i$-th element of $S(x;\theta)$ by $S_i(x;\theta)$ for $i=1,\ldots,4$, namely, 
\[
 S_i(x;\theta)=\frac{f_{\theta_i}(x;\theta)}{f(x;\theta)}.
\]
Similarly, denote the $i$-th component of 
$S_\tau (x;\theta)$ and $\wt S(x;\theta)$ respectively by 
$S_{\tau,i}(x;\theta)$ and $\wt S_i (x;\theta)$
for $i=1,\ldots,4$. 
Notice that in the proofs, we often omit $i$ when the corresponding proof holds uniformly in $i$.
In such a case we make a caution in the beginning. 
\begin{lemma}
\label{lem:uniconti:s}
 For every $x$, $S_{\tau,i} (x;\theta)$ converges to $S_i(x;\theta),\,i=1,\ldots,4$ 
uniformly in $\theta \in C$ as $\tau \to 0$. 
\end{lemma}
\begin{proof}
Since the proof is valid for all $i$, 
we omit subscript $i$ of $\theta_i$ and write $S(x;\theta),\,S_\tau(x;\theta),\,f_\theta,\,f_{\tau,\theta}$ 
to indicate elements of them for convenience. 
 Since 
\begin{align}
\label{eq:stau1}
 S_{\tau}(x;\theta)-S(x;\theta) =
 \frac{f_{\tau,\theta}(x;\theta)-f_\theta(x;\theta)}{f_\tau(x;\theta)} +
 S (x;\theta)\frac{f(x;\theta)-f_\tau(x;\theta)}{f_\tau(x;\theta)}, 
\end{align}
we evaluate 
\begin{align*}
 |S_{\tau} (x;\theta)-S(x;\theta)| \le \Big|
\frac{f_{\tau,\theta}-f_\theta}{f}
\Big|+ |S|\Big|
\frac{f-f_\tau}{f}
\Big|. 
\end{align*}
Since $f$ is unimodal and the support is 
$\R$, $f>0$ holds on $\R$. Indeed $|f(x;\theta)|\sim c|x|^{-(1+\alpha)}$ 
for sufficiently large $|x|$. 
Due to Proposition 2.3 of \cite{matsui:2019}, 
$S(x;\theta)=O(\log |x|)$. 
Now in view of Lemma
 \ref{lem:tauapprox:derivatives} the left hand-side converges uniformly
 in $\theta\in C$ as $\tau \to 0$. 
\end{proof} 
The next proposition relies on \cite[Theorems 6.5 and 6.6]{brant:1984}, and 
for convenience we state them.
\begin{theorem}[Theorem 6.5 in \cite{brant:1984}]
\label{thm65:brant:1984}
 Let $h(x)$ be a differentiable function of period $2\pi/\tau$, such that $|h'(x)|<A$ for all $x$. Then
 there exists a trigonometric polynomial of degree $k$, $q(x)=a_0+\sum_{j=1}^k (a_j\cos j\tau x + b_j \sin j\tau x)$, 
 satisfying 
 \[
  \sup_{x\in \R} |q(x)-h(x)| \le KA /(k\tau), 
 \]
 where $K$ is an absolute constant. 
\end{theorem}
 Notice that Theorem \ref{thm65:brant:1984} 
 is based on \cite[Theorem 1, p.2]{jackson:1930}, where one could see more precisely that $K$ does not
 depend on any specialization of the corresponding functions. 
\begin{theorem}[Theorem 6.6 in \cite{brant:1984}]
\label{thm66:brant:1984}
 Let $h(x)$ be continuous, and let $r(x)$ and $s(x)$ be $k$th degree trigonometric polynomials in 
$x$ such that 
\begin{align}
\label{brant:thm66:c1}
  \sup_{x\in I_\tau} |h(x)-r(x)|<\delta 
\end{align}
and 
\begin{align}
\label{brant:thm66:c2} 
\int_{I_\tau} |h(x)-s(x)|^2 dx \le d. 
\end{align}
Then 
\[
 \sup_{x\in I_\tau} |h(x)-s(x)| \le 4 (k\tau d)^{1/2} + 5 \delta.
\]
\end{theorem}
\begin{proposition}
\label{prop:upperbounds}
There exists $\tau_0>0$ such that when $\tau \le \tau_0$ we could take 
an upper bound $A>0$ with which $|S_{\tau,i}'(x;\theta)|<A,\,i=1,\ldots,4$ hold 
uniformly in $\tau\le \tau_0$, $x\in \R$ and $\theta\in C$. 
Moreover, for any $i$,  
\begin{align}
\label{ineq:brantlike}
 \sup_{x\in I_\tau} |S_{\tau,i}(x;\theta)-\wt S_i(x;\theta)| \le KA
 \{4(\tau^{2+\alpha} k )^{-1/2} + 5(\tau k)^{-1} \},
\end{align}
where $K$ is a positive constant which does not depend on $\tau,\,k$
 and $\theta$. 
\end{proposition}

\begin{proof}
By the periodicity, it is enough to show that there exists $\tau_0>0$ such that 
$\sup_{\tau \le \tau_0} \sup_{x \in I_\tau,  \theta \in C} |S_{\tau,i}'(x;\theta)| <A$.
For convenience we omit subscript $i$ of $\theta_i$ and write, 
for instance, $f_\theta,\,f_{\tau,\theta},\,f_\theta'$ and $f_{\tau,\theta}'$
 since the proof is valid for all $i$. 
 Accordingly $i$ in $S_{\tau,i}$
 and $\wt S_{i}$ and their derivatives is also abbreviated. 
Firstly we consider $|S_{\tau,i}'(x;\theta)|$. 
We evaluate 
\[
 S_{\tau}'(x) = \frac{f_{\tau,\theta}'(x) f_\tau(x)}{f_\tau (x)^2} -
 \frac{f_\tau'(x)f_{\tau,\theta}(x)}{f_\tau(x)^2 } := J_{1}-J_{2}. 
\]
The numerator of $J_1$ is dominated by 
\begin{align*}
 |f_{\tau,\theta}' f_\tau | &\le |
 f_{\tau,\theta}'-f_\theta'||f_\tau-f|+ |
f_{\tau, \theta}'-f_{\theta}'|f
 + |f_\theta'||f_\tau-f|+ |f_\theta'f| \\
& \le c\big\{\tau^{3+2\alpha} \log 1/\tau + \tau^{2+\alpha} \log 1/\tau
 \cdot f
 +\tau^{1+\alpha} |f_\theta'|  + |f_\theta' f| \big\},
\end{align*}
where in the last step we use Lemma
 \ref{lem:tauapprox:derivatives}. Similarly the 
 numerator of $J_2$ is
 bounded by 
\begin{align*}
 |f_{\tau}' f_{\tau,\theta} | &\le |
 f_{\tau}'-f'||f_{\tau,\theta}-f_\theta|+ |f_{\tau}'-f'||f_\theta|
 + |f'||f_{\tau,\theta}-f_\theta|+ |f'f_\theta| \\
& \le c\big\{ \tau^{3+2\alpha} \log 1/\tau + \tau^{2+\alpha}|f_\theta|
 + \tau^{1+\alpha} \log 1/\tau \cdot |f'| + |f'f_\theta | \big\}.
\end{align*}
Next the denominator is evaluated. It follows by the definition that  
\[
 f_\tau(x)-f(x)= \sum_{j\neq 0} f(x+2\pi j/\tau).
\]
We recall that for sufficiently large $|x|$, there exists a uniform constant $\underline{c}>0$ in $\theta\in C$ such that 
$f(x;\theta) \ge \underline{c} |x|^{-(1+\alpha)}$. 
Thus, if $x\in I_\tau$, so that $\pi/\tau (2j-1) \le
 x+2\pi j/\tau \le \pi/\tau (2j+1)$, 
 we have 
 \[
  \sum_{j \neq 0} f(x+2\pi j/\tau) \ge \underline{c} \sum_{j\neq 0} |3\pi
 j/\tau|^{-(1+\alpha)} \ge c'\tau^{1+\alpha},
 \]
for sufficiently small $\tau$. 
Hence $f_\tau (x)-f(x) \ge c \tau^{1+\alpha}$ on $x\in I_\tau$, which yields 
\begin{align}
\label{ineq:denom}
 f_\tau(x;\theta)^2 \ge c (f(x;\theta)^2+\tau^{2(1+\alpha)}) \ge
 c' f(x;\theta) \tau^{1+\alpha}.
\end{align}
Now, correcting the above results, we obtain 
\[
 J_1 \le c\big(
\tau \log 1/\tau + |f_\theta'/f|
\big)\quad \text{and}\quad 
 J_2 \le c \big(\tau \log 1/\tau +\tau |f_\theta/f| +|f'/f^2 |\tau^{1+\alpha} \log 1/\tau+|f'f_\theta/f^2 | \big). 
\]
Then due to the tail behaviors of $f,f',f_\theta$ and $f_\theta'$ in
 Lemma 2.1 in \cite{matsui:2019}, for some small $\tau_0>0$ it is easy to find a uniform bound
 which does not depend on $\tau \le \tau_0$, $x \in I_\tau$, and even $\theta \in C$.  
 We specially notice that $|f'/f^2|=O(|x|^\alpha)$ and $|f'f_\theta /f^2|=O(|x|^{-1}\log |x|)$ for sufficiently large $|x|$. 
 Thus, we prove the existence of $A$ such that $|S_{\tau,i}'| \le A $ for
 all $i$. 

 For the inequality \eqref{ineq:brantlike}, we rely on Theorem \ref{thm66:brant:1984}
 and will check its conditions \eqref{brant:thm66:c1} and \eqref{brant:thm66:c2}. 
 In our case, $h(x)=S_\tau(x;\theta)$ and $s(x)=\wt S(x;\theta)$. 
 Due to Theorem \ref{thm65:brant:1984} and the result just
 proved, \eqref{brant:thm66:c1} holds with $\delta=KA/(k\tau)$, where 
 $r (x)$ is an unspecified trigonometric polynomial. For \eqref{brant:thm66:c2}
 let $m_\tau = \min_{x\in I_\tau} f_\tau(x;\theta)\ge c \tau^{1+\alpha}$
 and we obtain 
\begin{align*}
 \int_{I_\tau} (S_\tau(x;\theta)-\wt S(x;\theta))^2 dx &\le
 \frac{1}{m_\tau} \int_{I_\tau} |\wt S(x;\theta)-S_\tau(x;\theta)|^2
 f_\tau(x;\theta)dx \\
&\le \frac{1}{m_\tau} \int_{I_\tau} |r (x)-S_\tau(x;\theta)|^2
 f_\tau (x;\theta)dx \\
&\le \frac{1}{m_\tau} \Big|\frac{KA}{k\tau}\Big|^2 \le c \Bigg(
\frac{KA}{k\tau^{(3+\alpha)/2} }
\Bigg)^2,  
\end{align*}
where the second inequality follows from the fact that $\wt S(x;\theta)$ is
 the projection of $S_\tau(x;\theta)$ on $L^2_\tau (f)$. Now we get
 \eqref{ineq:brantlike} through Theorem \ref{thm66:brant:1984}. 
\end{proof}

\subsection{Proof of Theorem \ref{thm:consistency}}
\label{append:thm:consitency}

\begin{proof}[Proof of Theorem \ref{thm:consistency}]
Because $\delta$ can be arbitrarily small without loss of generality, we assume that $\bar{B}_\delta(\theta_0) \subset C$. 
Thus, we confine our parameter space to $\bar B_\delta(\theta_0)$ and show that any sequence $\tilde \theta_n$ such that $\tilde \theta_n \in \bar B_\delta (\theta_0)$ and $|\wt S_n(\tilde \theta_n)|=o_p(1)$ also satisfies $\tilde \theta_n \stackrel{p}{\to} \theta_0$.

Let $\wt S(\theta) =\E_{\theta_0}[\wt S(X; \theta)]$.
Then, by \cite[Theorem 5.7]{vandervaat:2000}, it is enough to show that there exist $\delta >0$ such that for any $0 < \varepsilon \le \delta/2$
\begin{align}
\label{def:uniform}
\inf_{\theta \in \bar{B}_\delta(\theta_0),\,|\theta-\theta_0| \ge
	\varepsilon} |\wt S(\theta)| > |\wt S(\theta_0)| =0
\end{align}
and that 
\begin{align}
\label{def:uniconv}
\sup_{\theta \in \bar{B}_\delta(\theta_0)} |\wt S_n(\theta)-\wt S(\theta)|
\stackrel{p}{\to} 0.
\end{align}
Notice that the TML estimator is included in the class of estimators we consider.
Indeed, let $\hat \theta_n =\arg \min_{\theta \in \bar B_\delta(\theta_0)}\,|\wt S_n(\theta)|$.
Then we see easily from \eqref{def:uniconv} that $|\wt S_n(\hat \theta_n)|\le |\wt S_n(\theta_0)| \stackrel{p}{\to} 
|\wt S(\theta_0)|=0$.

For any given $k$ and $\tau >0$, the condition \eqref{def:uniform} follows from Lemma \ref{lem:auxiliary1} and continuity of $\tilde{I}(\theta)$.
In case $\tau \to 0$, we have $\wt I (\theta)\to I(\theta)$ for all $\theta \in \bar B_\delta(\theta_0)$ 
if $k^{-1} = o(\tau^{2+\alpha_0+\delta})$, which is established at the end of the proof .
Because $I(\theta)$ is continuous and positive definite at $\theta_0$ (see the proof of \cite[Theorem 3.1]{matsui:2019}), 
we establish the first condition even when $\tau \to 0$.

The condition \eqref{def:uniconv} is satisfied if each element of $\wt S(x;\theta)$ is continuous in $\theta$ for every $x$ and is bounded by an integrable function (the envelope condition) \cite[p.46]{vandervaat:2000}.
As the proof is similar for any elements, we only establish the result for $\wt S_3(x;\theta)$, the third element of $\wt S(x;\theta)$. 
It is clear that $\wt S_3(x; \theta)$ is continuous in $\theta$ for any fixed $k$ and $\tau$.
Moreover, it follows from Proposition \ref{prop:upperbounds} and Lemma \ref{lem:uniconti:s} that $\wt S_3(x; \theta) \to S_3(x; \theta)$ 
uniformly over $\theta \in \bar B_\delta(\theta_0)$ if $\tau \to 0$ and $k^{-1}=o(\tau^{2+\alpha_0+\delta})$.
Since $S_3(x; \theta)$ is continuous in $\theta$, so is $\wt S_3(x; \theta)$ even when $\tau \to 0$.
For the envelope condition, we use the following inequality
\begin{align}
 \label{wtS:decomp1}
 |\wt S_3 (x;\theta)| \le |S_{\tau,3} (x;\theta)| +|\wt S_3(x;\theta)-S_{\tau,3}(x;\theta)|.
\end{align}
The second term, which is periodic with period $2\pi/\tau$, has the same bound as that of Proposition \ref{prop:upperbounds}.
Thus, if $k^{-1}=o(\tau^{2+\alpha_0+\delta})$, it is bounded for any $\theta \in \bar B_\delta(\theta_0)$. 
For the first term, we have
\begin{align}
\label{wtS:decomp2}
|S_{\tau,3}(x;\theta)| \le \frac{|f_\alpha (x;\theta)|}{f(x;\theta)} +
\frac{\sum_{j\neq 0}
	|f_\alpha (x+2\pi j/\tau;\theta)|}{f_{\tau}(x;\theta)}
=:|S_3(x;\theta)|+ R_3(x;\theta). 
\end{align}
The first part satisfies $O(\log |x|)$ for sufficiently large $|x|$
(see \cite[Proposition 2.3]{matsui:2019}).
Since $S_{\tau,3}(x;\theta)$ is periodic, we consider the bound of $R_3(x;\theta)$ on $I_\tau$.

The interval $I_\tau$ is further divided into two regions: $|x|\le M$ and $M<|x| \le \pi/\tau$, 
where the constant $M$ is chosen such that 
$\underline{c} x^{-(1+\alpha)} \le f(x;\theta) \le \ov{c}
x^{-(1+\alpha)}$ holds for all $|x| >M$ and all $\theta\in C$. 
This is possible due to the tail property of stable laws (see, e.g. \cite[Proposition 1.2.15]{samorodnitsky:taqqu:1994}). 
Here, $\underline{c} \le \ov{c}$ are constants that do not depend on $\theta$. Since 
the tail of $|f_\alpha(x;\theta)|$ is bounded by $c|x|^{-(1+\alpha)} \log |x|$, which is decreasing on $|x|>M\, (\ge e)$, 
the 
numerator of $R_3$ is bounded on $|x|\le \pi/\tau$ by 
\begin{align*}
c \sum_{j\neq 0} |x+2\pi j/\tau|^{-(1+\alpha)} \log |x+2\pi
j/\tau| 
& \le c \sum_{j\neq 0} |xj|^{-(1+\alpha)} \log |xj| \\
& \le c |x|^{-(1+\alpha)} \log |x|, 
\end{align*}
while the denominator has the lower bound $f(x; \theta) \ge c |x|^{-(1+\alpha)}$ on $|x|>M$. 
Thus, 
\begin{align}
\label{supR1}
\sup_{\theta \in C} R_3 (x;\theta)\le c {\bf 1}_{\{|x|\le M\}} +c'\log |x|{\bf 1}_{\{|x| > M \}}. 
\end{align}
This together with the bound for $|S_3(x;\theta)|$ yields 
\begin{align}
\label{supstau}
\sup_{\theta \in C}|S_{\tau,3}(x;\theta)| \le c +c'\log |x| 
\end{align}
on $I_\tau$.
Since $S_{\tau,3}$ is periodic, the bound holds for all $x \in \R$. 
Thus we obtain
\[
\int_{-\infty}^\infty \sup_{\theta\in C}|S_{\tau,3}(x;\theta)| f(x;\theta_0)dx \le c+c'\int_{|x|>M} |x|^{-(1+\alpha_0)} \log |x| dx <\infty. 
\]

Finally we prove $\wt I(\theta)\to I(\theta)$. Since the proofs are similar, we only show the convergence for the information of $\alpha$, that is,
$\wt I_\alpha(\theta)\to I_\alpha (\theta)$. Since $\wt S_3(x;\theta)\to S_3(x;\theta)$ pointwise in $x$, in view of 
\[
\wt I_\alpha(\theta) = \int \wt S_3(x;\theta)^2 f(x;\theta)dx, 
\]
it suffices to show that $\wt S_3(x;\theta)^2$ has an integrable dominant function
w.r.t. the measure $f(x;\theta)dx$. 
Observe that 
\begin{align*}
|\wt S_3 (x;\theta)|^2 &\le  2 \big(S_{\tau,3} (x;\theta)^2 +|\wt S_3(x;\theta)-S_{\tau,3}(x;\theta)|^2 \big), 
\end{align*}
where the second term has the bound \eqref{ineq:brantlike}. For the first term, we notice that \eqref{supstau} holds for 
all sufficiently small $\tau>0$ and the bound does not depend on $\tau>0$ (see the derivation process also). Thus, there exists a $\tau_0>0$ such that 
for all $\tau\le \tau_0$, 
if $k^{-1}=o(\tau^{2+\alpha_0+\delta})$, $ \sup_{\theta\in \bar B_\delta(\theta_0)} |\wt S_{3}(x;\theta)|^2 \le (c+c'\log |x|)^2$ holds. 
Hence $\wt S_3(x;\theta)^2$ has a dominant integrable function. 
\end{proof}

\subsection{Proof of Theorem \ref{thm:asympeffi}}
\label{subsec:lemma:asymptoticnormal}
We first show $\sqrt{n}$-consistency of $\hat \theta_n$ by using a similar argument with the proof of \cite[Theorem 5.21]{vandervaat:2000}. 
Here we use the fact that $\wt S(x; \theta)$ can be arbitrarily close to the score function $S(x;\theta)$ of the ML estimator, and establish the result by using the properties of $S(x;\theta)$.
The condition on $k$ and $\tau$ assures us that $\wt S$ and $S$ are close enough for our purpose. 
To obtain the asymptotic distribution of $\hat \theta_n$, we show that $n^{1/2}(\hat \theta_n-\tilde \theta_n) \stackrel{p}{\to} 0$, where $\tilde{\theta}_n$ is the ML estimator.
This result implies that the TML and ML estimators are first-order equivalent.

Before the proof we need the following lemma which gives 
a moment bound for the stochastic distance of $\wt S_n$ and $S_n$. The bound is proved to be a 
function of $k$ and $\tau$. 
The proof is rather technical, and we fully exploit the two-step approximation scheme of Section \ref{subsec:asymptQMLinfty}. 
\begin{lemma}
\label{lemma:asymptoticnormal}
 Let $B_\varepsilon(\theta_0)=\{\theta:|\theta-\theta_0|<\varepsilon\}$.
Then for arbitrary small $\varepsilon >0$, 
\begin{align*}
 \E_{\theta_0} \big [\sup_{\theta \in B_\varepsilon (\theta_0)}|\wt
 S(X;\theta)-S(X;\theta)| \big]\le c \tau^{\alpha_0} \log 1/\tau + c(\tau^{2+\alpha_0+\varepsilon}k)^{-1/2}. 
\end{align*}
\end{lemma}

\begin{proof}
 The left hand-side is bounded by 
\begin{align*}
\sum_{i=1}^4 \E_{\theta_0} \big[\sup_{\theta \in B_\varepsilon (\theta_0)}|\wt
 S_i (X;\theta)-S_i (X;\theta)|\big]. 
\end{align*}
We evaluate each element in the last sum  
and decompose them into the following three integrals,  
\begin{align*}
&\int_{-\infty}^\infty \sup_{\theta\in B_\varepsilon (\theta_0)}|
\wt S_i(x;\theta)-S_i(x;\theta)
|f(x;\theta_0) dx \\
& \le  \int_{-\infty}^\infty  \sup_{\theta\in B_\varepsilon (\theta_0)}|
\wt S_i(x;\theta)-S_{\tau,i} (x;\theta)
|f(x;\theta_0) dx \\
& \quad + \Big( \int_{I_\tau} dx +\int_{I_\tau^c} dx \Big) \sup_{\theta\in B_\varepsilon (\theta_0)}|
S_{\tau,i} (x;\theta)-S_i (x;\theta)
|f(x;\theta_0) \\
& =: J_1+(J_2+J_3). 
\end{align*}
For 
$J_1$ we observe that 
$|\wt S_i(x;\theta)-S_{\tau,i} (x;\theta)|$ 
is periodic and has the bound in \eqref{ineq:brantlike}. 
Thus taking $\sup_{\theta\in B_\varepsilon (\theta_0)}$ in \eqref{ineq:brantlike},
 we obtain $J_1 \le c (\tau^{2+\alpha_0+\varepsilon}k)^{-1/2}$. 

Turning to $J_2$, 
in view of \eqref{eq:stau1} it suffices to 
 evaluate 
\begin{align}
\label{def:R:R'}
 |S_{\tau,i} (x;\theta)-S_i (x;\theta)| & \le \frac{\sum_{j\neq 0}
 |f_{\theta_i} (x+2\pi j/\tau;\theta)|}{f_\tau (x;\theta)} +
|S_i (x;\theta)| \cdot
 \frac{\sum_{j\neq 0} 
f(x+2\pi j/\tau ;\theta)}{f_\tau
 (x;\theta)}  \\
& =: R_i(x;\theta)+R'_i(x;\theta).\nonumber
\end{align}
Although $R_i (x;\theta)$ is treated in the proof of Theorem
 \ref{thm:consistency}, we use another bound here.   
Recall that $x^{-(1+\alpha)}\log x,\,(x>e)$ is
 decreasing. Since $\pi(2|j|-1)/\tau \le |x+2\pi j/\tau| \le
 \pi(2|j|+1)/\tau,\,j\neq 0$, the numerator has a bound 
\begin{align*}
& c \sum_{j\neq 0} |x+2\pi j/\tau |^{-(1+\alpha)} \log |x+2\pi
 j/\tau| \\
& \le c \tau^{1+\alpha} \sum_{j\neq 0} |\pi (2|j|-1)|^{-(1+\alpha)} \log |\pi
 (2|j|-1)/\tau| \\
& \le c \tau^{1+\alpha} \log 1/\tau, 
\end{align*}
for a sufficiently small $\tau>0$. 
In the meanwhile, the denominator has a uniform lower bound on $|x|\le M$, and on
 $M<|x| \le \pi/\tau$ it satisfies 
\[
 f_\tau(x;\theta) \ge f(x;\theta) \ge c |x|^{-(1+\alpha)},
\]
where the constant $M$ is the one in the proof of Theorem \ref{thm:consistency}. 
Thus, we have 
\begin{align*}
 R_i(x;\theta) \le \left \{
\begin{array}{ll}
c \tau^{1+\alpha} \log 1/\tau  & \text{on}\ |x|\le M\\
c' |\tau x/\pi|^{1+\alpha} \log 1/\tau  & \text{on}\ M<|x| \le \pi/\tau, 
\end{array}
\right.
\end{align*}
and since $|\tau x/\pi| \le 1$, we further obtain 
\begin{align}
\label{bound:R}
\sup_{\theta\in B_\varepsilon (\theta_0)} R_i(x;\theta) \le \left \{
\begin{array}{ll}
c \tau^{1+\alpha_0-\varepsilon} \log 1/\tau  & \text{on}\ |x|\le M\\
c' |\tau x/\pi|^{1+\alpha_0-\varepsilon} \log 1/\tau  & \text{on}\
 M<|x|\le \pi/\tau. 
\end{array}
\right.
\end{align}
Now, 
\begin{align*}
 \int_{I_\tau} \sup_{\theta \in B_\varepsilon (\theta_0)} R_i (x;\theta) f(x;\theta_0)
 dx &\le c \tau^{1+\alpha_0 -\varepsilon} \log 1/\tau \int_{|x|\le
 M}f(x;\theta_0) dx \\
& \quad + c' \tau^{1+\alpha_0 -\varepsilon} \log 1/\tau \int_{M<|x|\le
 \pi/\tau} |x/\pi|^{1+\alpha_0-\varepsilon} f(x;\theta_0) dx \\
& \le c \tau^{1+\alpha_0 -\varepsilon} 
\log 1/\tau\Big(1+ \int_{M<|x|\le
 \pi/\tau} |x/\pi|^{-\varepsilon} dx \Big) \\
& \le c\tau^{\alpha_0} \log 1/\tau. 
\end{align*}
Concerning $R'_i(x;\theta)$, we observe that on $I_\tau$, 
\begin{align}
\label{bound:R'}
 \sup_{\theta\in B_\varepsilon (\theta_0)} |S_i(x;\theta)|\le c \log |x|
 \le c \log 1/\tau,
\end{align}
while by similar calculations to that of $R_i$, 
\begin{align*}
 \left|
\frac{\sum_{j\neq 0} f(x+2\pi j/\tau;\theta) }{f_\tau(x;\theta)}
\right| \le \left \{
\begin{array}{ll}
c \tau^{1+\alpha}  & \text{on}\ |x|\le M\\
c' |\tau x/\pi|^{1+\alpha}   & \text{on}\
 M<|x| \le \pi/\tau. 
\end{array}
\right.
\end{align*}
Thus, except for constants, $R'_i$ has the same uniform bound as \eqref{bound:R}. 
Now collecting above results, we reach 
$$
J_2 \le c \tau^{\alpha_0} \log
 1/\tau.
$$

We again use the bound \eqref{def:R:R'} for $J_3$. 
For $R_i$ of $J_3$ we recall the argument in the proof of Theorem \ref{thm:consistency}. 
Define 
\[
 |S_{\tau,i}|(x;\theta)= \sum_{j=-\infty}^\infty \frac{|f_{\theta_i}(x+2\pi j/\tau;\theta)|}{f_\tau(x;\theta)} = 
\frac{|f_{\theta_i} (x;\theta)|}{f_\tau(x;\theta)}+ R_i(x;\theta)
\]
and then 
\[
 R_i (x;\theta) \le  |S_{\tau,i}|(x;\theta) +|S_i (x;\theta)|. 
\]
Notice that $|S_{\tau,i}|(x;\theta)$ is periodic and by exactly the same way as that for $S_{\tau,\alpha}(x;\theta)$, 
$|S_{\tau,i}|(x;\theta)$ has the bound $c \log |x|$ on $I_\tau^c$ (cf. \eqref{supstau}). Since $S_i (x;\theta) =O(\log|x|)$ for large $|x|$, 
$\sup_{\theta \in B_\vep(\theta_0)}R_i (x;\theta) \le c\log |x|$.
Due to \eqref{bound:R'}
 together with $R'_i\le |S|$, 
we further obtain $\sup_{\theta \in
 B_\varepsilon (\theta_0)} R'_i(x;\theta) \le c' \log |x|$. Thus 
\begin{align*}
 J_3 &\le \int_{I_\tau^c} 
\sup_{\theta \in B_\varepsilon (\theta_0)} \big(R_i(x;\theta)+R'_i(x;\theta)\big) f(x;\theta_0)
 dx \\ 
 & \le c \int_{I_\tau^c}  
|x|^{-(1+\alpha_0)}\log |x| dx \\
 & \le c \tau^{\alpha_0} \log 1/\tau. 
\end{align*}
Therefore, we obtain the result.
\end{proof}

Now we are ready to prove Theorem \ref{thm:asympeffi}. 
\begin{proof}[Proof of Theorem \ref{thm:asympeffi}]
We start to see $\wt S_n(\hat \theta_n)=o_p(n^{1/2})$. 
Let $\tilde \theta_n$ be the ML estimator and let $D_n:=\{
\omega: S_n(\tilde \theta_n)=0 \}$.
Then, $\P (D_n) \to 1$ as $n \to \infty$. 
By the definition of $\hat \theta_n$ and the consistency of both $\hat \theta_n$ and 
$\tilde \theta_n$, we observe that  
\begin{align*}
& \P\,\big(|n^{1/2} \wt S_n(\hat \theta_n)| > \delta \big) \le \P\big(|n^{1/2}\wt S_n(\tilde \theta_n)|>\delta \big) \\
&\le \P\,\Big (\big\{n^{1/2}|\wt S_n(\tilde \theta_n)-S_n(\tilde \theta_n)|>\delta \big\} \cap 
D_n \Big)+\P\,\big(D_n^c\big) \\
 &\le
 \P\,\Big(
\big\{n^{1/2}|\wt S_n(\tilde \theta_n)-S_n(\tilde \theta_n)|>\delta \big\} \cap 
\{\tilde \theta_n\in B_\varepsilon(\theta_0)\}
\cap D_n \Big) +\P\,\big(\tilde \theta_n \notin B_\varepsilon(\theta_0)\big)+\P\,\big(D_n^c\big) \\
 &\le \delta^{-1} \E_{\theta_0} \,\big[ 
 n^{1/2} |\wt S_n(\tilde \theta_n) -S_n(\tilde \theta_n) |{\bf 1}_{\{\tilde \theta_n \in B_\varepsilon
(\theta_0) \}}
\big]+o(1) \\
& \le   \delta^{-1} n^{1/2} \E_{\theta_0} \big[
\sup_{\theta \in B_\vep (\theta_0)} |\wt S(X;\theta)-S(X;\theta)|
\big]+o(1).
\end{align*}
Due to Lemma \ref{lemma:asymptoticnormal} and the condition on $\tau,\,k$, the right-hand side converges to $0$ as 
$n\to \infty$.

Next we prove 
\begin{align}
\label{conv:donsker}
 n^{1/2}\big(
\wt S_n(\hat \theta_n)- \wt S(\hat \theta_n)\big) -n^{1/2}\big(
\wt S_n(\theta_0)- \wt S(\theta_0)
\big)\stackrel{p}{\to} 0, 
\end{align}
where $\wt S(\theta)$ is the function $\theta \mapsto \E_{\theta_0}[\wt S(X;\theta)]$ (This corresponds to \cite[(5.22)]{vandervaat:2000}). 
Write the left-hand side as 
\begin{align*}
&  \underbrace{n^{1/2}\big\{
\wt S_n(\hat \theta_n)- S_n(\hat \theta_n) -\big(
\wt S(\hat \theta_n)- S(\hat \theta_n) \big) \big\}}_{I_1} + \underbrace{n^{1/2}\big\{S_n(\hat \theta_n)- S(\hat \theta_n) -\big(
S_n(\theta_0)- S(\theta_0) \big) \big\}}_{I_2} 
 \\
& - \underbrace{n^{1/2}\big\{\wt S_n(\theta_0)- S_n(\theta_0) -\big(
\wt S(\theta_0)- S(\theta_0) \big) \big\}}_{I_3} =: I_1+I_2-I_3, 
\end{align*}
where $S(\theta)$ is a function $\theta \mapsto \E_{\theta_0}[S(X;\theta)]$.
We study $n^{1/2}\big(
\wt S_n(\hat \theta_n)- S_n(\hat \theta_n)\big)$ in $I_1$ first. Observe that for any $\delta>0$,
\begin{align*}
 \P\,\big(n^{1/2}|\wt S_n(\hat \theta_n)- S_n(\hat \theta_n) |>\delta\big) 
&= \P\, \big(\{n^{1/2}|\wt S_n(\hat \theta_n)- S_n(\hat \theta_n) |>\delta\} \cap \{\hat \theta_n \in B_\varepsilon (\theta_0)\}\big)+ 
\P\,\big(\hat \theta_n \notin B_\varepsilon (\theta_0)\big) \\
&\le  
\delta^{-1} \E_{\theta_0} \big[
n^{1/2}|
\wt S_n(\hat \theta_n) -S_n(\hat \theta_n)
|{\bf 1}_{\{\hat \theta_n \in B_\varepsilon
(\theta_0) \}}
\big]
+o(1) \to 0 
\end{align*}
as $n\to \infty$, where we borrow the logic of the previous proof in the last step. 
In a similar manner, other quantities in $I_1$ and $I_3$ converge to $0$ in probability. 
For the quantity $I_2$, we use the continuous differentiability of the score of MLE: 
$\theta \mapsto S(x;\theta)$ for every $x$. In view of Lemma 2.1, in \cite{matsui:2019} each 
element of the matrix $S_\theta(x;\theta)=\partial S(x;\theta)/\partial \theta$ is continuous in $x$ and 
$\| S_\theta (x;\theta) \|=O((\log |x|)^2)$ in the tail for all $\theta\in C$, where $\|\cdot\|$ a matrix norm. 
Hence $\E\big[\sup_{\theta\in C}\|S_\theta(X;\theta)\|^2\big]<\infty$, and 
the Lipschitz condition for $S(x;\theta)$ given in \cite[Theorem 5.21]{vandervaat:2000} is satisfied, where $\dot \psi(x)$ there is 
replaced with $\sup_{\theta\in C}\|S_\theta (x;\theta)\|$ here. 
Then due to \cite[Example 19.7]{vandervaat:2000}, the functions $S(x;\theta)$ form a Donsker class 
(see also the proof of \cite[Theorem 5.21]{vandervaat:2000}) and by \cite[Lemma 19.24]{vandervaat:2000}, $I_2 \stackrel{p}{\to}0$.

Recall that $\wt S_n(\hat \theta_n)=o_p(n^{1/2})$ and $\wt S(\theta_0)=0$, and we may write 
\begin{align*}
 n^{1/2}\big(
\wt S_n(\hat \theta_n)- \wt S(\hat \theta_n)\big) &= n^{1/2}\big(
\wt S(\theta_0)- \wt S(\hat \theta_n)
\big) +o_p(1) \\
&= n^{1/2}\big(
\wt S(\theta_0)- S(\theta_0)
\big) + n^{1/2}\big(
S(\theta_0)- S(\hat \theta_n)
\big) - n^{1/2}\big(
\wt S(\hat \theta_n)- S(\hat \theta_n)
\big) +o_p(1) \\
&= n^{1/2}\big(S(\theta_0)- S(\hat \theta_n)\big) +o_p(1),
\end{align*}
where we again use Lemma \ref{lemma:asymptoticnormal} in the last step.

Now applying the delta method to the last quantity, we find from \eqref{conv:donsker} that 
\[
 n^{1/2} I(\theta_0)(\hat \theta_n- \theta_0) +n^{1/2}o_p(|\hat \theta_n-\theta_0|) = n^{1/2}\big(
\wt S_n(\theta_0) -\wt S(\theta_0)
\big)+o_p(1),
\] 
where $I(\theta_0)$ is the Fisher information, that is, the derivative of $\theta \mapsto -\E_{\theta_0}[S(X;\theta)]$ at $\theta_0$ 
(Here $I(\theta_0)$ corresponds to $-V_{\theta_0}$ in \cite[Theorem 5.21]{vandervaat:2000}). 
Since $I(\theta_0)$ is invertible (see the beginning of the proof for Theorem \ref{thm:consistency}), we have 
\[
 n^{1/2}|\hat \theta_n-\theta_0| \le \|I(\theta_0)^{-1}\| n^{1/2}| I(\theta_0)(\hat \theta_n-\theta_0) | 
=O_p(1)+ o_p(n^{1/2}|\hat \theta_n-\theta_0|).
\]
This implies $\hat \theta_n$ is $n^{1/2}$-consistent, so that 
$n^{1/2}o_p(|\hat \theta_n-\theta_0|)=o_p(1)$. Thus we obtain 
\[
 n^{1/2}I(\theta_0) (\hat \theta_n-\theta_0) = n^{1/2} \wt S_n(\theta_0)+o_p(1). 
\]
Finally, by the asymptotic linear representation of the ML estimator \cite[Theorem 3.1]{matsui:2019}, we have
\begin{align*}
 n^{1/2}(\hat \theta_n-\tilde \theta_n) &=n^{1/2} (\hat \theta_n-\theta_0) -n^{1/2} (\tilde \theta_n-\theta_0)+o_p(1) \\   
& =I(\theta_0)^{-1} n^{1/2}\big(\wt S_n(\theta_0)-S_n(\theta_0)\big)+o_p(1). 
\end{align*}
Again by Lemma \ref{lemma:asymptoticnormal} the left is $o_p(1)$. 
\end{proof}

\section{Supplement for Section \ref{application:ou}}
\label{derivatives:oustable}
Exact forms of derivative vectors of the conditional ch.f. 
$\varphi_{t\mid s,\theta}(u;\theta)$ with
$\theta=(\lambda,\alpha,\sigma)$ are given. For convenience we  
define $\psi_{t\mid s}(u;\theta)=\log
\varphi_{t\mid s}(u;\theta)$ and present expressions for $\psi_{t \mid s,\theta_i}(u;\theta)=\partial
\psi_{t\mid s}(u;\theta)/\partial \theta_i, i=1,2,3$, which are  
\begin{align*}
\psi_{t\mid s,\lambda}(u;\theta) &= -iu (t-s)X_s e^{-\lambda(t-s)}
 +\frac{|\sigma u|^\alpha}{\lambda \alpha }\big(
\lambda^{-1}(1-e^{-\alpha \lambda(t-s)})-\alpha(t-s) e^{-\alpha \lambda (t-s)}
\big), \\
\psi_{t \mid s,\sigma}(u;\theta) &= -
 \sigma^{\alpha-1}|u|^\alpha \lambda^{-1} (1-e^{-\alpha \lambda (t-s)}),\\
\psi_{t \mid s,\alpha}(u;\theta) &= 
\frac{|\sigma u|^\alpha}{\lambda \alpha}
\big\{
 (-\log |\sigma u|+\alpha^{-1})(1-e^{-\lambda \alpha
 (t-s)})-\lambda(t-s) e^{-\lambda \alpha (t-s)}
\big\}.
\end{align*}


{\small
\begin{thebibliography}{99}
\baselineskip12pt

\bibitem{brant:1984} Brant, R. (1984) 
Approximate likelihood and probability calculations based on transforms. 
{\it The Annals of Statistics} {\bf 12}, 989--1005. 

\bibitem{carrasco:chrnov:florens:ghysels:2007}
{\sc Carrasco, M., Chernov, M., Florens, J.P. and Ghysels, E.} (2007). 
Efficient estimation of general dynamic models with a continuum of
	moment conditions. 
{\it Journal of Econometrics} {\bf 140}, 529--573.

\bibitem{carrasco:florens:2000}
{\sc Carrasco, M. and Florens, J.P.} (2000) 
Generalization of GMM to a continuum of moment conditions. 
{\it Econometric Theory} {\bf 16}, 797--834.

\bibitem{carrasco:florens:2002}
{\sc Carrasco, M. and Florens, J.P.} (2002) 
Efficient GMM estimation using the empirical characteristic function.
Working paper; Department of Econometrics: University of Rochester. 

\bibitem{carrasco:kotchoni:2017} 
{\sc Carrasco, M. and Kotchoni, R.} (2017)
Efficient estimation using the characteristic function. 
{\it Econometric Theory} {\bf 33}, 479--526.

\bibitem{chacko:viceira:2003}
{\sc Chacko, G. and Viceira, L.M.} (2003)
Spectral GMM estimation of continuous-time processes. 
{\it Journal of Econometrics} {\bf 116}, 259--292.

\bibitem{Christensen:2006}
{\sc Christensen, O. and Christensen, K.L.} (2006) 
Linear independence and series expansions in function spaces. 
{\it The American Mathematical Monthly} {\bf 113}, 611--627.

\bibitem{clement:gloter:2019}
{\sc Cl\'ement, E. and Gloter, A.} (2019)  
Estimating functions for SDE driven by stable L\'evy processes. 
{\it Annales de l'Institut Henri Poincar\'e, Probabilit\'es et Statistiques} {\bf 55}, (2019), 1316--1348.

\bibitem{clement:gloter:2020}
{\sc Cl\'ement, E. and Gloter, A.} (2020) 
Joint estimation for SDE driven by locally stable L\'evy processes. 
{\it Electronic Journal of Statistics} {\bf 14}, 2922--2956. 

\bibitem{donald:imbens:newey:2003}
{\sc Donald S.G., Imbens, G.W. and Newey, W.K.} (2003)
Empirical likelihood estimation and consistent tests with conditional moment restrictions,
{\it Journal of Econometrics} {\bf 117}, 55--93. 


\bibitem{dumouchel:1973}
{\sc DuMouchel, W.H.} (1973) On the asymptotic normality of the maximum-likelihood estimate
when sampling from a stable distribution.
{\it The Annals of Statistics} {\bf 1}, 948--957.

\bibitem{dumouchel:1975}
{\sc DuMouchel, W.H.} (1975)
Stable distributions in statistical inference: 2. Information from stably distributed samples.
{\it Journal of the American Statistical Association} {\bf 70}, 386--393.

\bibitem{feuerverger:mcDunnough:1981a}
{\sc Feuerverger, A. and McDunnough, P.} (1981) 
On the efficiency of empirical characteristic function procedures. 
{\it Journal of the Royal Statistical Society. Series B (Methodological)} {\bf 43}, 20--27.

\bibitem{feuerverger:mcDunnough:1981b}
{\sc Feuerverger, A. and McDunnough, P.} (1981) 
On some Fourier methods for inference. 
{\it Journal of the American Statistical Association} {\bf 76}, 379--387.

\bibitem{hansen:1982}
{\sc Hansen, L.P.} (1982)
Large sample properties of generalized method of moments estimators.
{\it Econometrica} {\bf 50}, 1029--1054. 


\bibitem{hansen:heaton:yaron:1996}
{\sc Hansen, L.P., Heaton, J. and Yaron, A.} (1996)
Finite-sample properties of some alternative GMM estimators.
{\it Journal of Business \& Eonomic Statistics} {\bf 14}, 262--280.

\bibitem{heathcote:1977}
{\sc Heathcote, C.R.} (1977)
The integrated squared error estimation of parameters. 
{\it Biometrika} {\bf 64}, 255--264.

\bibitem{heyde:1997}
{\sc Heyde, C. C.} (1997)
{\it Quasi-Likelihood And Its Application: A General Approach to Optimal Parameter Estimation},
Springer, New York.

\bibitem{imbens:spady:johnson:1998}
{\sc Imbens, G. W., Spady, R. H. and Johnson, P.} (1997)
Information theoretic approaches to inference in moment condition models
{\it Econometrica} {\bf 66}. 333--357.

\bibitem{jackson:1930}
{\sc Jackson, D.} (1930)
{\it Theory of Approximation}. A.M.S. Colloq. Pub. {\bf 11}. 


\bibitem{kitamura:stuzer:1997}
{\sc Kitamura, Y. and Stuzer, M.}
An information-theoretic alternative to generalized method of moments estimation
{\it Econometrica} {\bf 65}, 861--874.


\bibitem{kogon:williams:1998}
{\sc Kogon, S.M. and Williams, D.B.} (1998). 
Characteristic function based estimation of stable distribution parameters. 
In: {\it A Practical Guide to Heavy Tails: Statistical Techniques and Applications},
Birkh\"auser, Boston, 311--338.

\bibitem{kolmogorov:formin:2012}
{\sc Kolmogorov A.N. and Fomin S.V.} (2012) 
{\it Elements of the Theory of Functions and Functional Analysis [Two
	Volumes in One]}, Martino Publishing, Eastford.

\bibitem{koutrouvelis:1980}
{\sc Koutrouvelis, I.A.} (1980)
Regression-type estimation of the parameters of stable laws. 
{\it Journal of the American Statistical Association} {\bf 75}, 918--928.

\bibitem{koutrouvelis:1981}
{\sc Koutrouvelis, I.A.} (1981) 
An iterative procedure for the estimation of the parameters of stable laws. 
{\it Communications in Statistics-Simulation and Computation} {\bf 10}, 17--28.

\bibitem{kunitomo:owada:2006}
{\sc Kunitomo, N. and Owada, T.} (2006). 
Empirical likelihood estimation of L\'evy processes. 
Graduate School of Economics, University of Tokyo Discussion Paper.

\bibitem{masuda:2019}
{\sc Masuda, H.} (2019). 
Non-Gaussian quasi-likelihood estimation of SDE driven by locally stable L\'evy process. 
{\it Stochastic Processes and their Applications} {\bf 129}, 1013--1059.

\bibitem{matsui:2019}
{\sc Matsui, M.} (2020)
Asymptotics of maximum likelihood estimation for stable law with
	continuous parameterization. 
{\it Communications in Statistics - Theory and Methods} (online)

\bibitem{matsui:2005}
{\sc Matsui, M.} (2005)
Fisher information matrix of general stable distributions close to the
	normal distribution. 
{\it Mathematical Methods of Statistics} {\bf 14}, 224--251. 

\bibitem{matsui:takemura:2006}
{\sc Matsui, M. and Takemura, A.} (2006) 
Some improvements in numerical evaluation of
symmetric stable density and its derivatives.
{\it Communications in Statistics - Theory and Methods} {\bf 35}, 149--172. 

\bibitem{matsui:takemura:2008}
{\sc Matsui, M. and Takemura, A.} (2008) 
Goodness-of-fit tests for symmetric stable distributions 
-- empirical characteristic function approach. 
{\it Test} {\bf 17}, 546--566.

\bibitem{mcculloch:1986}
{\sc McCulloch, J.H.} (1986)
Simple consistent estimators of stable distribution parameters.
{\it Communications in Statistics - Simulation andComputation} {\bf 15}, 1109--1136.


\bibitem{mittniketal:1999}
{\sc Mittnik, S., Doganoglu, T., and Chenyao, D.} (1999). 
Maximum likelihood estimation of stable Paretian models. 
{\it Mathematical and Computer Modelling} {\bf 29}, 275--293.

\bibitem{nagaev:shkolnik:1988}
{\sc Nagaev, A.V. and Shkol'nik, S.M.} (1988)
Some properties of symmetric stable distributions close to
the normal distribution.
{\it Theory of Probability \& Its Applications} {\bf 33}, 139--144.



\bibitem{newey:smith:2004}
{\sc Newey, W. K. and Smith, R. J.}
Higher order properties of GMM and generalized empirical likelihood estimators.
{\it Econometrica} {\bf 72}, 219--255.

\bibitem{nolan:1998}
{\sc Nolan, J.P.} (1998) Parameterizations and modes of stable distributions.
{\it Statistics \& Probability Letters} {\bf 38}, 187--195.

\bibitem{nolan:2001}
{\sc Nolan, J.P.} (2001)
{\it Maximum likelihood estimation and diagnostics for stable
distributions}.
In: {\it L\'evy Processes: Theory and Applications} 
(O.~E.~Barndorff-Nielsen {\it et al}. eds.),
Birkh\"auser, Boston, 379--400.

\bibitem{nolan:2020}
{\sc Nolan, J. P.} (2020). 
{\it Univariate Stable Distributions: Models for Heavy Tailed Data}. 
Springer, Cham.


\bibitem{qin:lawless:1994}
{\sc Qin, J. and Lawless, J.} (1994)
Empirical Likelihood and General Estimating Equations.
{\it The Annals of Statistics} {\bf 22}, 300--325.

 
\bibitem{paulsoetal:1975}
{\sc Paulson, A.S., Holcomb, E.W. and Leitch, R.A.} (1975) 
The estimation of the parameters of the stable laws.
{\it Biometrika} {\bf 62}, 163--170.

\bibitem{press:1972}
{\sc Press, S.J.} (1972) 
Estimation in univariate and multivariate stable distributions.
{\it Journal of the American Statistical Association} {\bf 67}, 842--846. 



\bibitem{royuela:at:al:2017}
{\sc Royuela-del-Val, J., Simmross-Wattenberg, F., and Alberola-L\'opez, C.} (2017). 
libstable: Fast, Parallel, and High-Precision Computation of $\alpha$-Stable Distributions in R, C/C++, and MATLAB. 
{\it Journal of Statistical Software} {\bf 78}, 1--23

\bibitem{samorodnitsky:taqqu:1994}
{\sc Samorodnitsky, G. and Taqqu, M.S.} (1994)
{\it Stable Non-Gaussian Random Processes.
Stochastic Models with Infinite Variance.} Chapman and Hall, London.


\bibitem{sato:1999}
{\sc Sato, K.} (1999) 
{\it L\'evy Processes and Infinitely Divisible Distributions.} 
Cambridge University Press, Cambridge.

\bibitem{singleton:2001}
{\sc Singleton, K.J.} (2001) 
Estimation of affine asset pricing models using the empirical
	characteristic function. 
{\it Journal of Econometrics} {\bf 102}, 111--141.

\bibitem{sueishi:nishiyama:2005}
{\sc Sueishi, N. and Nishiyama, Y.} (2005)  
{\it Estimation of Levy Processes in Mathematical Finance: A Comparative Study}. 
In: {\it MODSIM 2005 International Congress on Modelling and Simulation}.  
(Zerger, A. and Argent, R.M. eds.)
 Modelling and Simulation Society of Australia and New Zealand, December 2005, 953--959. 

\bibitem{vandervaat:2000}
{\sc van der Vaart, A.W.} (2000)
{\it Asymptotic Statistics (Cambridge Series in Statistical and
	Probabilistic Mathematics).}
Cambridge University Press, Cambridge.

\bibitem{yu:2004}
{\sc Yu, J.} (2004)\  
Empirical characteristic function estimation and its applications. 
{\it Econometric Reviews} \textbf{23}, 93--123.

\bibitem{zolotarev}
{\sc Zolotarev,\ V.M.\ } (1986)\ \textit{One-Dimensional
     Stable Distributions.} 
    AMS Translation\ of Mathematics\ Monographs,\ \textbf{65},\ American\ Mathematics\  Society, Providence.\ (Transl.\ of the original 1983 Russian)
\end{thebibliography}}
\end{document}